\documentclass[11pt,a4paper]{scrartcl}
\usepackage[utf8]{inputenc}
\usepackage[english]{babel}
\usepackage{helvet, stmaryrd, ulem}
\usepackage{amsmath}
\usepackage{amsthm}
\usepackage{amsfonts}
\usepackage{amssymb}
\usepackage{mathtools}
\usepackage{mathrsfs}
\usepackage{dsfont}
\usepackage{tabularx}
\usepackage{xcolor}
\usepackage{bbm}
\usepackage[paper=a4paper,left=25mm,right=25mm,top=25mm,bottom=35mm]{geometry}

\usepackage{hyperref}
\usepackage[noabbrev,capitalize]{cleveref}

\author{Helena Kremp, Nicolas Perkowski}
\title{Periodic homogenization for singular SDEs}

\newtheorem{theorem}{Theorem}[section]
\newtheorem{definition}[theorem]{Definition}     
\newtheorem{proposition}[theorem]{Proposition}
\newtheorem{lemma}[theorem]{Lemma}	

\newtheorem{corollary}[theorem]{Corollary}

\newtheorem{remark}[theorem]{Remark}
\newtheorem{assumption}[theorem]{Assumption}

\numberwithin{equation}{section}

\newcommand{\R}{\mathbb{R}}

\newcommand{\N}{\mathbb{N}}
\newcommand{\E}{\mathds{E}}
\newcommand{\p}{\mathds{P}}

\newcommand{\F}{\mathcal{F}}
\newcommand{\La}{\mathcal{L}^{\alpha}_{\nu}}

\newcommand{\calC}{\mathscr{C}}

\newcommand{\calX}{\mathscr{X}}

\renewcommand{\mathcal}[1]{\mathscr{#1}}
\renewcommand{\epsilon}{\varepsilon}

\newcommand{\para}{\varolessthan}
\newcommand{\arap}{\varogreaterthan}
\newcommand{\reso}{\varodot}

\newcommand{\squeeze}[2][0]{\mbox{$\medmuskip=#1mu\displaystyle#2$}}

\DeclarePairedDelimiter{\abs}{\lvert}{\rvert}
\DeclarePairedDelimiter{\norm}{\lVert}{\rVert}

\DeclarePairedDelimiter{\paren}{(}{)}

\begin{document}

\begin{center}
\begin{huge}
Periodic homogenization for singular Lévy SDEs\\
\end{huge}
\begin{Large}
\vspace{0.5cm}
Helena Kremp\footnote{Technische Universität Wien, helena.kremp@asc.tuwien.ac.at}, Nicolas Perkowski\footnote{Freie Universität Berlin, perkowski@math.fu-berlin.de}
\end{Large}
\vspace{1cm}
\hrule
\end{center}
\begin{section}*{Abstract}
We generalizes the theory of periodic homogenization for multidimensional SDEs with additive Brownian and stable Lévy noise for $\alpha\in (1,2)$ (cf.~\cite{Bensoussan1978, Franke2007}) to the setting of singular periodic Besov drifts $F\in(\calC^{\beta}(\mathbb{T}^{d}))^{d}$ for $\beta\in ((2-2\alpha)/3,0)$ beyond the Young regime. For the martingale solution from \cite{kp} projected onto the torus, we prove existence and uniqueness of an invariant probability measure $\pi$ with strictly positive Lebesgue density exploiting the theory of paracontrolled distributions and a strict maximum principle for the singular Fokker-Planck equation. 
Furthermore, we prove a spectral gap on the semigroup of the diffusion and solve the Poisson equation with singular right-hand side equal to the drift itself. 
In the CLT scaling, we prove that the diffusion converges in law to a Brownian motion with constant diffusion matrix. In the pure $\alpha$-stable noise case, we rescale in the scaling that the $\alpha$-stable process respects and show convergence to the stable process itself. We conclude on the periodic homogenization result for the parabolic PDE for the singular generator  $\mathfrak{L}^{\epsilon}=-(-\Delta)^{\alpha/2}+\epsilon^{1-\alpha}F(\epsilon^{-1}\cdot)\cdot\nabla$ as $\epsilon\to 0$.
\end{section}
\vspace{0,5cm}
\hrule
\vspace{1cm}
\begin{section}{Introduction}
\noindent Periodic homogenization decribes the limit procedure from microscopic boundary-value problems posed on periodic structures to a macroscopic equation. Such periodic media are for example composite materials or polymer structures. The theory originated from engineering purposes in material sciences in the 1970s, cf.~\cite{Bensoussan1978} and the references therein. Mathematically, this leads to the study of the limit of periodic operators with rapidly oscillating coefficients. There exist analytic and probabilistic methods to determine the limit equation. We refer to the classical works \cite{Bensoussan1978, Pavliotis2008} for the background on homogenization theory. We employ a probabilistic method using the Feynman-Kac formula (cf.~\cite{Oksendal2003}). Via the Feynman-Kac formula, the periodic homogenization result for the Kolmogorov PDE with fluctuating and unbounded drift corresponds to a central limit theorem for the diffusion process.\\
\noindent In this work, we generalize the theory of periodic homogenization for SDEs with additive Brownian noise, respectively stable Lévy noise, from \cite{Bensoussan1978}, respectively \cite{Franke2007}, from the setting of regular coefficients to singular Besov drifts $F\in(\calC^{\beta}(\mathbb{T}^{d}))^{d}$ for $\beta\in ((2-2\alpha)/3,0)$ on the $d$-dimensional torus $\mathbb{T}^{d}$.\\ 
In \cite[Section 3.4.2]{Bensoussan1978}, the periodic drift coefficient is assumed to be $C^{1}$ with Hölder-continuous derivative and the periodic diffusion coefficient is assumed to be symmetric and uniformly elliptic, as well as $C^{2}$ with Hölder-continuous first derivative and bounded second derivative. The assumption of uniform ellipticity can be relaxed to allow for some degeneracy, which was investigated in \cite{Hairer2008} using Malliavin calculus techniques.\\  
In \cite{Franke2007} the multiplicative symmetric $\alpha$-stable noise case for $\alpha\in (1,2)$ is studied and the coefficients are assumed to be even more regular, namely $C^{3}$. The regularity assumptions were relaxed in \cite{Huang2018, Huang2022}, where the authors more generally consider the periodic homogenization for the generator of an $\alpha$-stable-like Feller process. In \cite{Huang2022}, using a Zvonkin transformation to remove the drift (cf.~\cite{Zvonkin1974}), the authors can consider drifts that are bounded and $\beta$-Hölder continuous for $\beta\in (1-\alpha/2,1)$. They also consider a non-linear intensity function $\sigma$ and therefore a multiplicative noise term of the form $\sigma(X_{t},dL_{t}^{\alpha})$, see \cite[Equation (2.1)]{Huang2022} with an isotropic $\alpha$-stable process $L^{\alpha}$, whereas in \cite{Franke2007} the intensity function $\sigma(x,y)$ is linear in $y$.\\ In the recent article \cite{Chen2021} the authors further generalize the assumption on the drift coefficient to bounded, measurable drifts and consider the solution of the martingale problem associated to the SDE. The operator they consider is a Lévy-type operator that in particular includes all stable Lévy noise generators, symmetric and non-symmetric. 
They prove the homogenization result with the corrector method, an analytical method in homogenization theory, and show that different limit phenomena occur in the cases $\alpha\in (0,1)$, $\alpha=1$, $\alpha\in (1,2)$, $\alpha=2$ and $\alpha\in (2,\infty)$.\\
With analytical methods, the papers \cite{Krassmann, Arisawa, Schwab} deal with Lévy-type operators with oscillating coefficients for $\alpha\in (0,2)$, but without drift part.\\
In the mixed jump-diffusion case, \cite{Sandric} investigates the periodic homogenization for zero-drift diffusions with small jumps. The homogenized process in this case is also a Brownian motion.\\
We focus on the addive $\alpha$-stable symmetric noise case, where different limit behaviours occure for $\alpha=2$ (the Brownian noise case) and $\alpha\in (1,2)$. Our contribution is the generalization to distributional drifts, not only in the Young, but also in the rough regime.\\
For the homogenization result, we rely on Kipnis-Varadhan martingale methods (cf.~\cite{Kipnis1986} and \cite{klo}). Those methods require to solve the Poisson equation for the generator of the diffusion (or more generally the resolvent equations and imposing additional assumption) and to rewrite the additive functional in terms of that solution and Dynkin's martingale. Poisson equations for generators of diffusions with regular coefficients were studied in the classical article \cite{Pardoux1998}.\\ 
Following \cite{klo}, we generalize those techniques to much less regular drift coefficient. In particular this includes bounded measurable drifts or distributional drifts in the Young regime, where classical PDE techniques apply. More interestingly, our theory applies in the setting of singular drifts such as a typical realization of the periodic spatial white noise, cf. \cref{rem:p-Brox}. In order to apply the SDE solution theory from \cite{kp}, we restrict to additive noise.\\
To be more precise, we study the functional central limit theorem for the solution $X$ of the martingale problem associated to the SDE 
\begin{align}\label{eq:sde}
dX_{t}=F(X_{t})dt+dL_{t}
\end{align} with $F\in(\calC^{\beta}(\mathbb{T}^{d}))^{d}$ and a symmetric $\alpha$-stable process $L$ for $\alpha\in (1,2]$. The singular generator $\mathfrak{L}$ of $X$ is given by 
\begin{align*}
\mathfrak{L}=-\La+F\cdot\nabla.
\end{align*}
The first step is to prove existence and uniqueness of an invariant probability measure $\pi$ on $\mathbb{T}^{d}$ for $\mathfrak{L}$ with strictly positive Lebesgue density. We achieve this by solving the singular Fokker-Planck equation with singular initial condition $\mu\in \calC^{0}_{1}$,
\begin{align*}
(\partial_{t}-\mathfrak{L}^{\ast})\rho_{t}=0,\quad \rho_{0}=\mu,
\end{align*} with formal Lebesgue adjoint $\mathfrak{L}^{\ast}$ of $\mathfrak{L}$ and proving a strict maximum principle on compacts. Furthermore, we prove spectral gap estimates for the semigroup of the diffusion projected onto the torus and solve the singular resolvent equation for $\mathfrak{L}$. This enables, through a limiting argument in a Sobolev-type space $\mathcal{H}^{1}(\pi)$ with respect to $\pi$, to solve the Poisson equation \eqref{eq:Poisson-eq} with singular right-hand side $F-\langle F\rangle_{\pi}$. Here, we define $\langle F\rangle_{\pi}=\int F d\pi$ in a stable manner.\\
For the homogenization, we distinguish between the cases $\alpha=2$ (Brownian noise case) and $\alpha\in (1,2)$, as the scaling and the limit behaviour differs. 
In the standard Brownian noise case, we prove weak convergence
\begin{align}\label{eq:mainr1}
\paren[\bigg]{\frac{1}{\sqrt{n}}(X_{nt}-nt\langle F\rangle_{\pi})}_{t\in[0,T]}\Rightarrow (\sqrt{D}B_{t})_{t\in[0,T]},
\end{align}
where $B$ is a standard Brownian motion and $D$ is the constant diffusion matrix with entries
\begin{align*}
D(i,j):=\int_{\mathbb{T}^{d}} (e_{i}+\nabla \chi^{i}(x))(e_{j}+\nabla \chi^{j}(x))^{T}\pi(dx),
\end{align*} for $i,j=1,\dots,d$ and $e_{i}$ denoting the $i$-th euclidean unit vector. The limit is motivated by the result from \cite[Section 3.4.2]{Bensoussan1978}. 
Furthermore, $\chi\in (L^{2}(\pi))^{d}$ solves the Poisson equation with singular right-hand side $F-\langle F\rangle_{\pi}$: 
\begin{align}\label{eq:Poisson-eq}
(-\mathfrak{L})\chi^{i}=F^{i}-\langle F^{i}\rangle_{\pi},
\end{align} for $i=1,\dots,d$. 
In the pure Lévy noise case $\alpha\in (1,2)$ we rescale in the $\alpha$-stable scaling $n^{-1/\alpha}$ instead of $n^{-1/2}$. In this scaling we show, that the Dynkin martingale 
vanishes and thus we obtain weak convergence towards the stable process itself,
\begin{align}\label{eq:mainr2}
\paren[\bigg]{\frac{1}{n^{1/\alpha}}(X_{nt}-nt\langle F\rangle_{\pi})}_{t\in [0,T]}\Rightarrow (L_{t})_{t\in[0,T]}.
\end{align} 
In particular, compared to the Brownian noise case, there is no diffusivity enhancement in the limit (analogously to the regular coefficient case, cf.~\cite{Franke2007}).

\noindent The paper is structured as follows. Preliminaries and the strategy to prove the central limit theorem are outlined in \cref{sec:prelim-PH}. In \cref{sec:s-FP} we solve the singular Fokker-Planck equation with the paracontrolled approach. 
The singular resolvent equation for $\mathfrak{L}$ is solved in \cref{sec:res-eq}. We show in \cref{sec:inv-m} existence and uniqueness of the invariant measure $\pi$. 
\cref{sec:inv-m} furthermore yields a characterization of the domain of the generator $\mathfrak{L}$ in $L^{2}(\pi)$, cf.~\cref{thm:generator}. In \cref{sec:s-P-eq}, we solve the Poisson equation with singular right-hand side $F-\langle F\rangle_{\pi}$. Finally, we prove the CLT in \cref{sec:CLT} and relate to the periodic homogenization result for the parabolic PDE with oscillating operator $\mathfrak{L}^{\epsilon}=-\La+\epsilon^{1-\alpha}F(\epsilon^{-1}\cdot)\cdot\nabla$, cf.~\cref{cor:PDEhomog}. 
\end{section}

\begin{section}{Preliminaries}\label{sec:prelim-PH}
This section gives an introduction to periodic Besov spaces and Schauder and exponential Schauder estimates on such. Furthermore, we introduce the projected solution $X^{\mathbb{T}^{d}}$ of $X$ onto the torus and its generator $\mathfrak{L}$ and semigroup. We define the space of enhanced distributions $\mathcal{X}^{\beta,\gamma}_{\infty}$. This section finishes with a summary on our strategy in proving the convergence results \eqref{eq:mainr1} and \eqref{eq:mainr2}.\\

\noindent Let $\mathcal{S}(\R^{d})$ be the space of Schwartz functions and $\mathcal{S}'(\R^{d})$ the space of tempered distributions. A periodic (or $1$-periodic) distribution $u$ satisfies $u(\varphi(\cdot+1))=u(\varphi)$ for all $\varphi\in\mathcal{S}(\R^{d})$. Let $\mathbb{T}^{d}=(\R/\mathbb{Z})^{d}$ denote the torus and $\mathcal{S}(\mathbb{T}^{d})$ the space of Schwartz functions on the torus, i.e. smooth functions with the locally convex topology generated by the family of semi-norms $\norm{f}_{\gamma}=\sup_{x\in\mathbb{T}^{d}}\abs{D^{\gamma}f(x)}$ for multi-indices $\gamma\in\N^{d}$, and its topological dual $\mathcal{S}'(\mathbb{T}^{d})$.
Let $(p_{j})_{j\geqslant -1}$ be a smooth dyadic partition of unity, i.e. a family of functions $p_{j}\in C^{\infty}_{c}(\R^{d})$ for $j\geqslant -1$, such that 
\begin{enumerate}
\item[1.)]$p_{-1}$ and $p_{0}$ are non-negative radial functions (they just depend on the absolute value of $x\in\R^{d}$), such that the support of $p_{-1}$ is contained in a ball and the support of $p_{0}$ is contained in an annulus;
\item[2.)]$p_{j}(x):=p_{0}(2^{-j}x)$, $x\in\R^{d}$, $j\geqslant 0$;
\item[3.)]$\sum_{j=-1}^{\infty}p_{j}(x)=1$ for every $x\in\R^{d}$; and
\item[4.)]$\operatorname{supp}(p_{i})\cap \operatorname{supp}(p_{j})=\emptyset$ for all $\abs{i-j}>1$.
\end{enumerate} 
Then, we define the Besov space on the torus with regularity $\theta\in\R$, integrability $p\in[1,\infty]$ and summability $q\in[1,\infty)$ as (cf.~\cite[Section 3.5]{Schmeisser1987}) 
\begin{align}
B^{\theta}_{p,q}(\mathbb{T}^{d}):=\{u\in\mathcal{S}'(\mathbb{T}^{d})\mid \norm{u}_{B^{\theta}_{p,q}}:=\norm{(2^{js}\norm{\Delta_{j}u}_{L^{p}(\mathbb{T}^{d})})_{j\geqslant -1}}_{l^{q}(\mathbb{Z}^{d})}<\infty\}
\end{align}
for the Littlewood-Paley blocks $\Delta_{j}u=\F^{-1}_{\mathbb{T}^{d}}(\rho_{j}\F_{\mathbb{T}^{d}}u)$ with Fourier transform $\F_{\mathbb{T}^{d}}f(k)=\hat{f}(k)=\int_{\mathbb{T}^{d}}f(x)e^{-2\pi i k \cdot x}dx$, $k\in\mathbb{Z}^{d}$, with inverse Fourier transform $\F^{-1}_{\mathbb{T}^{d}}f(x)=\sum_{k\in\mathbb{Z}^{d}}f(k)e^{2\pi i k\cdot x}$ and a dyadic partition of unity $(\rho_{j})_{j\geqslant -1}$ as above (cf. also~\cite[Section 3.4.4]{Schmeisser1987}). The Fourier transform on $\R^{d}$, we define as $\F f(y)= \int_{\mathbb{T}^{d}}f(x)e^{-2\pi i y \cdot x}dx$, $y\in\R^{d}$ with inverse $\F^{-1}f(y)=\F f(-y)$, $f\in\mathcal{S}(\R^{d})$. In the case $q=\infty$, we rather work with the separable Besov space, and thus define
\begin{align}\label{eq:Besov-torus-infty}
B^{\theta}_{p,\infty}=B^{\theta}_{p,\infty}(\mathbb{T}^{d}):=\{u\in\mathcal{S}'(\mathbb{T}^{d})\mid \norm{u}_{B^{\theta}_{p,\infty}}:=\lim_{j\to\infty} 2^{js}\norm{\Delta_{j}u}_{L^{p}}=0\}.
\end{align} 
We introduce the notation $\calC^{\theta}_{p}(\mathbb{T}^{d}):=B^{\theta}_{p,\infty}(\mathbb{T}^{d})$ for $p\in[1,\infty)$ and $\calC^{\theta}(\mathbb{T}^{d}):=B^{\theta}_{\infty,\infty}(\mathbb{T}^{d})$ and analogously for $\mathbb{T}^{d}$ replaced by $\R^{d}$. We simply write $B^{\theta}_{p,q}$, respectively $\calC^{\theta}_{p}$, in the case the statement holds for any of the spaces, on the torus $\mathbb{T}^{d}$ and on $\R^{d}$.
We recall from Bony's paraproduct theory (cf. \cite[Section 2]{Bahouri2011}) that in general for $u\in\calC^{\theta}$ and $v\in\calC^{\beta}$ with $\theta,\beta\in\R$, the product $u v:=u\para v+u\arap v +u \reso v$ , is  well defined in $\calC^{\min(\theta,\beta,\theta+\beta)}$ if and only if $\theta+\beta>0$. Denoting $S_{i}u=\sum_{j=-1}^{i-1}\Delta_{j}u$, the paraproducts are defined as follows
\begin{align*}
u\para v:=\sum_{i\geqslant -1} S_{i-1}u\Delta_{i}v,\quad u\arap v:=v\para u, \quad u\reso v:= \sum_{\abs{i-j}\leqslant 1}\Delta_{i}u\Delta_{j}v.
\end{align*} Here, we use the notation of \cite{Martin2017, Mourrat2017Dynamic} for the para- and resonant products $\para, \arap$ and  $\reso$.\\
In estimates we often use the notation $a\lesssim b$, which means, that there exists a constant $C>0$, such that $a\leqslant C b$. In the case that we want to stress the dependence of the constant $C(d)$ in the estimate on a parameter $d$, we write $a\lesssim_{d} b$.\\
Let $C^{\infty}_{b}=C^{\infty}_{b}(\R^{d},\R)$ denote the space of smooth, bounded functions with bounded partial derivatives.\\
The paraproducts satisfy the following estimates for $p,p_{1},p_{2}\in[1,\infty]$ with $\frac{1}{p}=\min(1,\frac{1}{p_{1}}+\frac{1}{p_{2}})$ and $\theta,\beta\in\R$ (cf. \cite[Theorem A.1]{PvZ} and \cite[Theorem 2.82, Theorem 2.85]{Bahouri2011})
\begin{equation}
\begin{aligned}\label{eq:paraproduct-estimates}
\norm{u\reso v}_{\calC^{\theta+\beta}_{p}} & \lesssim\norm{u}_{\calC^{\theta}_{p_{1}}}\norm{v}_{\calC^{\beta}_{p_{2}}}, \qquad \text{if }\theta +\beta > 0,\\
\norm{u\para v}_{\calC^{\beta}_{p}} \lesssim\norm{u}_{L^{p_{1}}}\norm{v}_{\calC^{\beta}_{p_{2}}}& \lesssim\norm{u}_{\calC^{\theta}_{p_{1}}}\norm{v}_{\calC^{\beta}_{p_{2}}}, \qquad \text{if } \theta > 0,\\
\norm{u\para v}_{\calC^{\beta+\theta}_{p}}& \lesssim\norm{u}_{\calC^{\theta}_{p_{1}}}\norm{v}_{\calC^{\beta}_{p_{2}}}, \qquad \text{if } \theta < 0.
\end{aligned}
\end{equation}
So if $\theta + \beta > 0$, we have $\norm{u v}_{\calC^{\gamma}_{p}}\lesssim\norm{u}_{\calC^{\theta}_{p_{1}}}\norm{v}_{\calC^{\beta}_{p_{2}}}$ for $\gamma:=\min(\theta,\beta,\theta+\beta)$.\\

\noindent Next, we collect some facts about $\alpha$-stable Lévy processes and their generators and semigroups. For $\alpha\in (0,2]$, a symmetric $\alpha$-stable Lévy process $L$ is a Lévy process, that moreover satisfies the scaling property $(L_{k t})_{t \geqslant 0}\stackrel{d}{=}k^{1/\alpha}(L_{t})_{t \geqslant 0}$ for any $k>0$ and $L\stackrel{d}{=}-L$, where $\stackrel{d}{=}$ denotes equality in law. These properties determine the jump measure $\mu$ of $L$, see \cite[Theorem 14.3]{Sato1999}.  That is, if $\alpha\in (0,2)$, the Lévy jump measure $\mu$ of $L$ is given by
\begin{align}\label{eq:mu}
\mu(A):=\E\bigg[\sum_{0\leqslant t\leqslant 1}\mathbf{1}_{A}(\Delta L_{t})\bigg]=\int_{S}\int_{\R^{+}}\mathbf{1}_{A}(k\xi)\frac{1}{k^{1+\alpha}}dk\tilde \nu(d\xi),\quad A\in \mathcal{B}(\R^{d}\setminus\{0\}),
\end{align} where $\tilde \nu$ is a finite, symmetric, non-zero measure on the unit sphere $S\subset\R^{d}$. 
Furthermore, we also define for $A\in\mathcal{B}(\R^{d}\setminus\{0\})$ and $t\geqslant 0$ the Poisson random measure
\begin{align*}
\pi(A\times [0,t])=\sum_{0\leqslant s\leqslant t}\mathbf{1}_{A}(\Delta L_{s}),
\end{align*} with intensity measure $dt\mu(dy)$. Denote the compensated Poisson random measure of $L$ by $\hat{\pi}(dr,dy):=\pi(dr,dy)-dr\mu(dy)$.
We refer to the book by Peszat and Zabczyk \cite{peszat_zabczyk_2007} for the integration theory against Poisson random measures. 
The generator $A$ of $L$ satisfies $C_{b}^{\infty}(\R^{d})\subset \operatorname{dom}(A)$ and is given by
\begin{align}\label{eq:functional}
A\varphi(x)=\int_{\R^{d}}\paren[\big]{\varphi(x+y)-\varphi(x)-\mathbf{1}_{\{\abs{y}\leqslant 1\}}(y) \nabla \varphi(x) \cdot y}\mu(dy)\qquad\text{for }\varphi\in C_{b}^{\infty}(\R^{d}).
\end{align} If $(P_t)_{t\geqslant 0}$ denotes the semigroup of $L$, the convergence $t^{-1}(P_{t}f(x)-f(x))\to Af(x)$ is uniform in $x\in\R^{d}$ (see \cite[Theorem 5.4]{peszat_zabczyk_2007}).

\noindent We also have a Fourier respresentation of the operator $A$, that is defined as follows.

\begin{definition}\label{def:fl}
Let $\alpha \in (0,2)$ and let $\nu$ be a symmetric (i.e. $\nu(A)=\nu(-A)$), finite and non-zero measure on the unit sphere $S\subset\R^{d}$. We define the operator $\La$ as
\begin{align}
\La\F^{-1}\varphi=\F^{-1}(\psi^{\alpha}_{\nu} \varphi)\qquad\text{for $\varphi\in C^\infty_b$,}
\end{align}  where
$\psi^{\alpha}_{\nu} (z):=\int_{S}\abs{\langle z,\xi\rangle}^{\alpha}\nu(d\xi).$
For $\alpha=2$, we set $\La:=-\frac{1}{2}\Delta$.\\
On the torus and for $\alpha\in (1,2)$, we define the fractional Laplacian as follows: for $f\in C^{\infty}(\mathbb{T}^{d})$ 
\begin{align*}
\La f=\F^{-1}_{\mathbb{T}^{d}}(\mathbb{Z}^{d}\ni k\mapsto\psi^{\alpha}_{\nu}(k)\hat{f}(k))
\end{align*} and for $\alpha=2$ analogously.
\end{definition}

\begin{remark}
If we take $\nu$ as a suitable multiple of the Lebesgue measure on the sphere, then $\psi^\alpha_\nu(z) = |2\pi z|^\alpha$ and thus $\La$ is the fractional Laplace operator $(-\Delta)^{\alpha/2}$. 
\end{remark}

\begin{lemma}\label{rem:lgen}\footnote{\cite[Lemma 2.3]{kp}}
Let $\alpha \in (0,2)$ and let again $\nu$ be a symmetric, finite and non-zero measure on the unit sphere $S\subset\R^{d}$. Then for $\varphi \in C^\infty_b$ we have $-\La \varphi = A\varphi$, where $A$ is the generator of the symmetric, $\alpha$-stable Lévy process $L$ with characteristic exponent $\E[\exp(2\pi i\langle z,L_{t}\rangle )]=\exp(-t\psi^{\alpha}_{\nu}(z))$. The process $L$ has the jump measure $\mu$ as defined in~\cref{eq:mu}, with $\tilde \nu = C \nu$ for some constant $C>0$.
\end{lemma}

\noindent If $\alpha=2$, then the generator of the symmetric, $\alpha$-stable process coincides with $\sum_{i,j}C(i,j)\partial_{x_{i}}\partial_{x_{j}}$ for an explicit covariance matrix $C$ (cf.~\cite[Theorem 14.2]{Sato1999}), that is, the generator of $\sqrt{2C}B$ for a standard Brownian motion $B$. To ease notation, we consider here $C=\frac{1}{2}\operatorname{Id}_{d\times d}$ and whenever we refer to the case $\alpha=2$, we mean the standard Brownian motion noise case and we set $\La:=-\frac{1}{2}\Delta$.

\begin{assumption}\label{ass}
Throughout the work, we assume that the measure $\nu$ from \cref{def:fl} has $d$-dimensional support, in the sense that the linear span of its support is $\R^d$. This means that the process $L$ can reach every open set in $\R^d$ with positive probability.
\end{assumption}

\noindent An $\alpha$-stable, symmetric Lévy process, that satisfies \cref{ass}, we also call non-degenerate.\\

\noindent In the following, we will not distinguish between $F\in(\calC^{\beta}(\mathbb{T}^{d}))^d$ and the periodic version on $\R^{d}$, $F^{\R^{d}}\in (\calC^{\beta})^d$, whenever there is no danger of confusion. 
We understand \eqref{eq:sde} as a singular SDE with periodic coefficient $F^{\R^{d}}$ and in particular existence of a solution to the martingale problem follows from \cite{kp}. For that, we need to assume that the drift $F$ can be enhanced in the following sense.
Let for $\gamma\in (0,1)$,
\begin{align*}
\mathcal{M}_{\infty,0}^{\gamma}X=\{u:(0,\infty)\to X\mid \exists C>0, \forall t>0, \norm{u_{t}}_{X}\leqslant C [t^{-\gamma}\vee 1]\}.
\end{align*}  
\begin{assumption}
For $\beta\in (\frac{2-2\alpha}{3},\frac{1-\alpha}{2}]$ we assume that $(F_{1}=F,F_{2})\in\mathcal{X}^{\beta,\gamma}_{\infty}(\mathbb{T}^{d})$, that is $(F^{\R^{d}},F_{2}^{\R^{d}})\in\mathcal{X}^{\beta,\gamma}_{\infty}$, where 
\begin{align}\label{def:enhanced-dist-ph}
\mathcal{X}^{\beta,\gamma}_{\infty}:=cl(\{\paren[\big]{\eta,(P_{t}(\partial_{i}\eta^{j})\reso\eta^{k})_{i,j,k\in\{1,\dots,d\}}}\mid \eta\in C^{\infty}_{b}(\R^{d},\R^{d})\})
\end{align} for $\gamma\in [(2\beta+2\alpha-1)/\alpha,1)$  and
for the closure in $\calC^{\beta+(1-\gamma)\alpha}\times \mathcal{M}_{\infty,0}^{\gamma}\calC^{2\beta+\alpha-1}$. For $\beta\in (\frac{1-\alpha}{2},0)$, we assume that $F\in\calC^{\beta+(1-\gamma)\alpha}$ for $\gamma\in ((\beta-1)/\alpha,0)$ and set $\calX^{\beta,\gamma}_{\infty}:=\calC^{\beta+(1-\gamma)\alpha}$.
\end{assumption}

\begin{remark}
The assumption on the enhanced distribution in \eqref{def:enhanced-dist-ph} is stronger than the assumption in \cite[Definition 4.2]{kp-sk} in the sense that $F$ is an enhanced distribution for any finite time horizon $T>0$, instead of for a fixed time horizon. This assumption will be needed in \cref{sec:res-eq} to solve the resolvent equation. Notice also that the blow-up $\gamma$ occures at the initial time $t=0$ and not at a terminal time and that $F$ does not depend on a time variable here. Furthermore, in the definition above we allow for three different indices $i,j,k$ in \eqref{def:enhanced-dist-ph}. This assumption is due to the fact that we also solve the adjoint equation, i.e. the Fokker-Planck equation. For the Fokker-Planck equation, we will encounter the products $P_{t}(\partial_{i}F^{i})\reso F^{j}$ for $i,j=1,\dots,d$, whereas for the Kolmogorov equation, we have $P_{t}(\partial_{i}F^{j})\reso F^{i}$ for $i,j$. To cover both products, we assume \eqref{def:enhanced-dist-ph}. The blow-up $\gamma$ can be thought of as close to $1$ and $t\mapsto P_{t}(\partial_{i}F^{j})\reso F^{k}\in\mathcal{M}^{\gamma}_{\infty,0}\calC^{2\beta+\alpha-1}$ in particular implies that for any $T>0$, $\int_{0}^{T}P_{t}(\partial_{i}F^{j})\reso F^{k}dt\in\calC^{2\beta+\alpha-1}$.
\end{remark}
\noindent For completeness we state the definition of a solution to the singular martingale problem from \cite[Definition 4.1]{kp}, cf. also \cite{Cannizzaro2018}, and \cite[Theorem 4.2]{kp} about the existence and uniqueness of martingale solutions.

\begin{definition}[Martingale problem]\label{def:martp}
Let $\alpha\in (1,2]$ and $\beta\in(\frac{2-2\alpha}{3},0)$, and let $T>0$ and $F^{\R^{d}}\in \mathcal{X}^{\beta,\gamma}_{\infty}$. Then, we call a probability measure $\p$ on the Skorokhod space $(\Omega,\mathcal{F})$ a solution of the martingale problem for $(\mathcal{G}^{V},\delta_x)$, if
\begin{enumerate}
\item[\textbf{1.)}] $\p(X_{0}\equiv x)=1$ (i.e. $\p^{X_{0}}=\delta_{x}$), and
\item[\textbf{2.)}] for all $f\in C_{T}\calC^{\epsilon}$ with $\varepsilon > 2-\alpha$ and for all $u^{T}\in\mathcal{C}^{3}$, the process $M=(M_{t})_{t\in [0,T]}$ is a martingale under $\p$ with respect to $(\mathcal{F}_{t})$, where
\begin{align}
M_{t}=u(t,X_{t})-u(0,x)-\int_{0}^{t}f(s,X_{s})ds
\end{align} and where $u$ is a mild solution of the Kolmogorov backward equation $\mathcal{G}^{F}u=f$ with terminal condition $u(T,\cdot)=u^{T}$, where $\mathcal{G}^{F}:=\partial_{t}-\La+F\cdot\nabla$.
\end{enumerate} 
\end{definition}
\begin{remark}
Although we consider a drift term $F$ that does not depend on a time variable, we consider the parabolic Kolmogorov PDE in the definition above. Equivalently one could reformulate the martingale problem with the resolvent equation for the operator $-\La+F\cdot\nabla$ instead. We use the above definition to be able to apply the result from \cite{kp}.
\end{remark}

\begin{theorem}\label{thm:mainthm1}
Let $\alpha\in (1,2]$ and $L$ be a symmetric, $\alpha$-stable L\'evy process, such that the measure $\nu$ satisfies \cref{ass}. Let $T>0$ and $\beta\in ((2-2\alpha)/3,0)$ and let $F^{\R^{d}}\in\calX^{\beta,\gamma}_{\infty}$. Then for all $x\in\R^{d}$, there exists a unique solution $\mathbb{Q}$ on $(\Omega,\mathcal F)$ of the martingale problem for $(\mathcal{G}^{V},\delta_x)$. Under $\mathbb{Q}$ the canonical process is a strong Markov process.
\end{theorem}

\noindent In the following, we will also consider the projected process $(X^{\mathbb{T}^{d}}_{t})=(\iota(X_{t}))$ for the canonical projection $\iota:\R^{d}\to\mathbb{T}^{d}$, $x\mapsto [x]=x\mod\mathbb{Z}^{d}$, and the martingale solution $X$ from \cref{thm:mainthm1}. The generator $\mathfrak{L}$ of $X^{\mathbb{T}^{d}}$ we define by
\begin{align*}
\mathfrak{L}f:=-\La f+F\cdot\nabla f
\end{align*} 
acting on functions $f:\mathbb{T}^{d}\to\R$.\\
This work moreover yields a characterization of the domain $\text{dom}(\mathfrak{L})$ of the generator $\mathfrak{L}$, cf.~\cref{thm:generator}. We denote its semigroup by $(T_{t}^{\mathbb{T}^{d}})_{t\geqslant 0}$ with $T_{t}^{\mathbb{T}^{d}}f:=T_{t}f^{\R^{d}}$, $f\in L^{\infty}(\mathbb{T}^{d})$, with the semigroup $(T_{t})_{t\geqslant 0}$ of the Markov process $(X_{t})$ on $\R^{d}$ with periodic drift $F^{\R^{d}}$.\\  
The semigroup $(P_{t}^{\mathbb{T}^{d}})$ of the generalized fractional Laplacian $(-\La)$ acting on functions on the torus, is analogously defined  as $P_{t}^{\mathbb{T}^{d}}f:=P_{t}f^{\R^{d}}$ and the semigroup estimates for $(P_{t})$ imply the estimates for $(P_{t}^{\mathbb{T}^{d}})$ on the periodic Besov spaces $\calC^{\theta}(\mathbb{T}^{d})=\calC^{\theta}_{\infty}(\mathbb{T}^{d})$ (due to $u\in L^{\infty}(\mathbb{T}^{d})$ implying $u^{\R^{d}}\in L^{\infty}(\R^{d})$ and vice versa). 
The following lemma states the semigroup estimates for $(P^{\mathbb{T}^{d}}_{t})$ on $\calC^{\theta}_{2}(\mathbb{T}^{d})$, that will be employed in the sequel. The proof can be found in \cref{Appendix A}. 
\cref{lem:periodic-semi-est} in particular proves the extension of $\La$ to Besov spaces $\calC^{\beta}_{2}(\mathbb{T}^{d})$.
\begin{lemma}\label{lem:periodic-semi-est}
Let $u\in\calC^{\beta}_{2}(\mathbb{T}^{d})$ for $\beta\in\R$. Then the following estimates hold true
\begin{align}\label{eq:p-la}
\norm{\La u}_{\calC^{\beta-\alpha}_{2}(\mathbb{T}^{d})}\lesssim\norm{u}_{\calC^{\beta}_{2}(\mathbb{T}^{d})}.
\end{align} 
Moeover, for any $\theta\geqslant 0$ and $\vartheta\in [0,\alpha]$,
\begin{align}\label{eq:p-semi}
\norm{P_{t}u}_{\calC^{\beta+\theta}_{2}(\mathbb{T}^{d})}\lesssim (t^{-\theta/\alpha}\vee 1)\norm{u}_{\calC^{\beta}_{2}(\mathbb{T}^{d})}, \quad \norm{(P_{t}-\operatorname{Id})u}_{\calC^{\beta-\vartheta}_{2}(\mathbb{T}^{d})}\lesssim t^{\vartheta/\alpha}\norm{u}_{\calC^{\beta}_{2}(\mathbb{T}^{d})}.
\end{align}
\end{lemma}

\noindent For functions with vanishing zero-order Fourier mode, we can improve the Schauder estimates for large $t>0$. This is established in the following lemma, the proof can be found in \cref{Appendix A}.

\begin{lemma}\label{lem:exp-schauder}
Let $(P_{t})$ be the $(-\La)$-semigroup on the torus $\mathbb{T}^{d}$ as defined above. Then for $g\in\calC^{\beta}_{2}$, $\beta\in\R$, with $\hat{g}(0)=\F_{\mathbb{T}^{d}}(g) (0)=0$, exponential Schauder estimates hold true. That is, for any $\theta\geqslant 0$, there exists $c>0$, such that 
\begin{align*}
\norm{P_{t}g}_{\calC^{\beta+\theta}_{2}(\mathbb{T}^{d})}\lesssim t^{-\theta/\alpha}e^{-ct}\norm{g}_{\calC^{\beta}_{2}(\mathbb{T}^{d})}.
\end{align*} 
\end{lemma}
\noindent In the sequel, we will employ the following duality result for Besov spaces on the torus. For Besov spaces on $\R^{d}$, the result is proven in \cite[Proposition 2.76]{Bahouri2011}. The same proof applies for Besov spaces on the torus (cf.~also \cite[Theorem in Section 3.5.6]{Schmeisser1987}).
\begin{lemma}\label{lem:duality}
Let $\theta\in\R$ and $f,g \in C^{\infty}(\mathbb{T}^{d})$. Then we have the duality estimate:
\begin{align}
\abs{\langle f, g\rangle}\lesssim \norm{f}_{B^{\theta}_{2,2}(\mathbb{T}^{d})}\norm{g}_{B^{-\theta}_{2,2}(\mathbb{T}^{d})}.
\end{align}
In particular, the mapping $(f,g)\mapsto \langle f,g\rangle$ can be extended uniquely to $f\in B^{\theta}_{2,2}(\mathbb{T}^{d})$, $g\in B^{-\theta}_{2,2}(\mathbb{T}^{d})$.
\end{lemma} 
Let us define the periodic Bessel-potential space or fractional Sobolev space for $s\in\R$,
\begin{align*}
H^{s}(\mathbb{T}^{d})=\biggl\{u\in \mathcal{S}'(\mathbb{T}^{d})\biggm| \norm{u}_{H^{s}(\mathbb{T}^{d})}^{2}=\sum_{k\in\mathbb{Z}^{d}}(1+\abs{k}^{2})^{s}\abs{\hat{f}(k)}^{2}<\infty\biggr\},
\end{align*} and the homogeneous periodic Bessel-potential space
\begin{align*}
\dot{H}^{s}(\mathbb{T}^{d})=\biggl\{u\in \mathcal{S}'(\mathbb{T}^{d})\biggm| \norm{u}_{H^{s}(\mathbb{T}^{d})}^{2}=\sum_{k\in\mathbb{Z}^{d}}\abs{k}^{2s}\abs{\hat{f}(k)}^{2}<\infty\biggr\}.
\end{align*}
\noindent Motivated by the corresponding characterization of periodic Besov spaces from \cite[Section 3.5.4]{Schmeisser1987}, we define the homogeneous Besov space on the torus for $\theta\in (0,1)$ with notation $\Delta_{h}u (x):=u(x+h)-u(x)$, $h,x\in\mathbb{T}^{d}$ as follows:
\begin{align}\label{eq:p-bs2}
\dot{B}^{\theta}_{2,2}(\mathbb{T}^{d}):=\biggl\{u\in L^{2}(\mathbb{T}^{d})\biggm| \norm{u}_{\dot{B}^{\theta}_{2,2}(\mathbb{T}^{d})}^{2}:=\int_{\mathbb{T}^{d}}\abs{h}^{-2\theta}\norm{\Delta_{h}u}_{L^{2}(\mathbb{T}^{d})}^{2}\frac{dh}{\abs{h}^{d}}<\infty\biggr\}.
\end{align}
For $\theta=1$, we set $\dot{B}^{1}_{2,2}(\mathbb{T}^{d}):=\dot{H}^{1}(\mathbb{T}^{d})$.
Using derivatives of $u$, one can define homogeneous periodic Besov spaces in that way also for $\theta\geqslant 1$ (cf. \cite[Section 3.5.4]{Schmeisser1987}), but we will not need them below.
We also refer to \cite[(iv) of Theorem, Section 3.5.4]{Schmeisser1987} for an equivalent characterization of spaces $B^{\theta}_{2,2}(\mathbb{T}^{d})$ for $\theta\in (0,1)$ in terms of the differences $\Delta_{h}u$.
 
\begin{subsection}*{Strategy to prove the main result}
To prove the CLT in \cref{thm:main-thm-ph}, we distinguish between the cases $\alpha=2$ and $\alpha\in (1,2)$.
In the following, we briefly summarize our strategy to prove the convergences \eqref{eq:mainr1} (Brownian case, cf.~\cite[Chapter 3, Section 4.2]{Bensoussan1978} in the case of $C^{2}_{b}$-drift) and \eqref{eq:mainr2} (pure Lévy noise case, cf.~\cite{Franke2007} in the case of $C^{3}_{b}$-drift).\\
First we prove existence of a unique invariant probability measure $\pi$ for $X^{\mathbb{T}^{d}}$. To that aim, we solve in \cref{sec:s-FP} the singular Fokker-Planck equation with the paracontrolled approach in $\calC^{\alpha+\beta-1}_{1}$, yielding a continuous (as $\alpha+\beta-1>0$) Lebesgue-density. Furthermore, we prove a strict maximum principle on compacts for the Fokker-Planck equation. In \cref{sec:inv-m} an application of Doeblin's theorem then yields existence and uniqueness of the invariant ergodic probability measure $\pi$ for $\mathfrak{L}$ with a strictly positive Lebesgue density $\rho_{\infty}$. Doeblin's theorem furthermore yields pointwise spectral gap estimates on the semigroup $(T_{t}^{\mathbb{T}^{d}})_{t\geqslant 0}$ associated to $\mathfrak{L}$, i.e.~the process $X^{\mathbb{T}^{d}}$ is exponentially ergodic. We then extend those pointwise spectral gap estimates to $L^{2}(\pi)$-spectral gap estimates. This enables to solve the Poisson equation in \cref{cor:good-P-eq} for right-hand sides that are elements of $L^{2}(\pi)$ and that have vanishing mean under $\pi$. In particular, we can solve the Poisson equation with right-hand side $F^{m}-\langle F^{m}\rangle_{\pi}$ for $F^{m}\in C^{\infty}(\mathbb{T}^{d})$ for each fixed $m\in\N$, where $F^{m}\to F$ in $\mathcal{X}^{\beta,\gamma}_{\infty}(\mathbb{T}^{d})$, denoting the solution by $\chi^{m}$. We then prove convergence of $(\chi^{m})_{m}$ in $L^{2}(\pi)$ utilizing a Poincaré-type estimate for the opertor $\mathfrak{L}$ and combining with the theory from \cite{klo}. Via solving the resolvent equation $(\lambda-\mathfrak{L})g=G$ in \cref{sec:res-eq} with the paracontrolled approach for right-hand-sides in $G\in L^{2}(\pi)$ or $G=F^{i}$, $i=1,...,d$, we then obtain in \cref{sec:s-P-eq} convergence of $(\chi^{m})_{m}$ in $(\calC^{\alpha+\beta}_{2}(\mathbb{T}^{d}))^{d}$ to a limit $\chi$ which indeed solves the Poisson equation $(-\mathfrak{L})\chi=F-\langle F\rangle_{\pi}$ with singular right-hand side $F-\langle F\rangle_{\pi}$. Here the mean $\langle F\rangle_{\pi}$ can be defined in a stable manner using the regularity, respectively the paracontrolled structure, of the density $\rho_{\infty}$, cf. \cref{thm:def-F-pi-int}.
Decomposing the drift in terms of the solution to the Poisson equation and Dynkin's martingale, we can finally prove the functional CLT in \cref{sec:CLT}.\\ 
Via Feynman-Kac formula, the CLT yields the periodic homogenization result of \cref{cor:PDEhomog} for the solution to the associated Cauchy problem with operator $\mathfrak{L}^{\epsilon}$ as $\epsilon\to 0$, where formally $\mathfrak{L}^{\epsilon}f=-\La f+\epsilon^{-1} F(\epsilon^{-1} \cdot)\cdot\nabla f$. 
\end{subsection}
\end{section}
\begin{section}{Singular Fokker-Planck equation and a strict maximum principle}\label{sec:s-FP}
This section features the results on the Fokker-Planck equation, \cref{thm:ex-fp-d} and \cref{prop:max-p}, that will be of use in \cref{sec:inv-m} below.\\
Let us define the blow-up spaces for $\gamma\in (0,1)$,
\begin{align*}
\mathcal{M}_{T,0}^{\gamma}X:=\bigl\{u:(0,T]\to X\bigm| \sup_{t\in[0,T]}t^{\gamma}\norm{u_{t}}_{X}<\infty\bigr\}
\end{align*} and 
\begin{align*}
C^{1,\gamma}_{T,0}X:=\biggl\{u:(0,T]\to X\biggm| \sup_{0\leqslant s<t\leqslant T}\frac{s^{\gamma}\norm{u_{t}-u_{s}}_{X}}{\abs{t-s}}<\infty\biggr\}
\end{align*} with blow-up at $t=0$.\\
The solution to the Fokker-Planck eqution with initial condition equal to a Dirac measure, will have a blow-up at time $t=0$ due to the singularity of the initial condition. 
A direct computation shows that the Dirac measure in $x\in\R^{d}$ satisfies $\delta_{x}\in\calC^{-d(1-\frac{1}{p})}_{p}$ for any $p\in[1,\infty]$, in particular $\delta_{x}\in\calC^{0}_{1}$. Moreover, one can show that the map $x\mapsto\delta_{x}\in\calC^{-\epsilon}_{1}$ is continuous for any $\epsilon>0$. The next theorem proves existence of a mild solution to the Fokker-Planck equation 
\begin{align*}
(\partial_{t}-\mathfrak{L}^{\ast})\rho_{t}=0, \quad \rho_{0}=\mu,
\end{align*} with initial condition $\mu\in\calC^{-\epsilon}_{1}$ for small $\epsilon>0$.
Here, $\mathfrak{L}^{*}$ denotes the formal Lebesgue-adjoint to $\mathfrak{L}$, 
\begin{align*}
\mathfrak{L}^{\ast}f:=-\La f-\nabla\cdot (Ff)=-\La f-div (Ff).
\end{align*}  
The proof of \cref{thm:ex-fp-d} is similar to \cite[Theorem 4.7]{kp-sk}. 
\begin{theorem}\label{thm:ex-fp-d}
Let $T>0$, $\alpha\in (1,2]$ and $p\in[1,\infty]$. Let either $\beta\in(\frac{1-\alpha}{2},0)$ and $F\in \calC^{\beta}_{\R^{d}}$ or $F\in\mathcal{X}^{\beta,\gamma'}_{\infty}$ for $\beta\in (\frac{2-2\alpha}{3},\frac{1-\alpha}{2}]$, $\gamma'\in (\frac{2\beta+2\alpha-1}{\alpha},1)$.\\ 
Then, for any small enough $\epsilon>0$ and any initial condition $\mu\in\calC^{-\epsilon}_{p}$, there exists a unique mild solution $\rho$ to the Fokker-Planck equation in $\mathcal{M}_{T,0}^{\gamma}\calC^{\alpha+\beta-1}_{p}\cap C_{T}^{1-\gamma}\calC^{\beta}_p\cap C_{T,0}^{1,\gamma}\calC^{\beta}_p$ for $\gamma\in (C(\epsilon),1)$ (for some $C(\epsilon)\in (0,1)$) in the Young regime and $\gamma\in (\gamma',\frac{\alpha\gamma'}{2-\alpha-3\beta})$ in the rough regime, i.e.
\begin{align}\label{eq:fp-mild-form}
\rho_{t}=P_{t}\mu+\int_{0}^{t}P_{t-s}(-\nabla\cdot (F\rho_{s}))ds,
\end{align}
where  $(P_{t})_{t\geqslant 0}$ denotes the $(-\La)$-semigroup.\\
In the rough case, the solution satisfies
\begin{align}
\rho_{t}=\rho_{t}^{\sharp}+\rho_{t}\para I_{t}(-\nabla \cdot F)
\end{align} where $\rho_{t}^{\sharp}\in\mathcal{M}_{T,0}^{\gamma}\calC^{2(\alpha+\beta)-2}_p\cap C_{T}^{1-\gamma}\calC^{2\beta-2+\alpha}_p\cap C_{T,0}^{1,\gamma}\calC^{2\beta-2+\alpha}_p$ and $I_{t}(v):=\int_{0}^{t}P_{t-s}v_{s}ds$.\\
Moreover, the solution depends continuously on the data $(F,\mu)\in\mathcal{X}^{\beta,\gamma'}_{\infty}\times\calC^{-\epsilon}_{p}$. Furthermore, for any fixed $t>0$, the solution satisfies $(\rho_{t},\rho_{t}^{\sharp})\in\calC^{\alpha+\beta-1}\times\calC^{2(\alpha+\beta)-2}$.\\
If $(F,\mu)$ are $1$-periodic distributions, then the solution $\rho_{t}$ is $1$-periodic.
\end{theorem}

\begin{proof}
We will prove that we can solve the Fokker-Planck equation for initial conditions $\mu\in \calC^{-\epsilon}_{p}$ for $\epsilon=-((1-\tilde{\gamma})\alpha+\beta)$ for $\tilde{\gamma}\in [\frac{\alpha+\beta}{\alpha},1)$ in the Young regime and for $\epsilon=-((2-\tilde{\gamma})\alpha+2\beta-1)$ for $\tilde{\gamma}\in [\frac{2\beta+2\alpha-1}{\alpha}\vee 0,\gamma']$ in the rough regime. In the Young regime, we obtain a solution $\rho\in \mathcal{M}_{T,0}^{\gamma}\calC^{\alpha+\beta-1}_p\cap C_{T}^{1-\gamma}\calC^{\beta}_p\cap C_{T,0}^{\gamma,1}\calC^{\beta}_p$ for $\gamma=\tilde{\gamma}$ and the proof is analogous to \cite[Theorem 4.1]{kp-sk}. 
We thus only give the proof in the rough regime.\\
To that aim, let us define, analogously as in the proof of \cite[Theorem 4.7]{kp-sk} for $\gamma\in (\gamma',1)$ as there, 
\begin{align*}
\mathcal{L}^{\gamma,\theta}_{T,p}:=\mathcal{M}_{T,0}^{\gamma}\calC^{\theta}_{p}\cap C_{T}^{1-\gamma}\calC^{\theta-\alpha}_{p}\cap C_{T,0}^{1,\gamma}\calC^{\theta-\alpha}_{p}
\end{align*} and the paracontrolled solution space
\begin{align*}
\squeeze[1]{\mathcal{D}^{\gamma}_{T,p}:=\{(u,u^{\prime})\in\mathcal{L}^{\gamma',\alpha+\beta-1}_{T,p}\times (\mathcal{L}^{\gamma,\alpha+\beta-1}_{T,p})^{d}\mid u^{\sharp}_{t}=u_{t}-u_{t}^{\prime}\para I_{t}(-\nabla\cdot F)\in\mathcal{L}^{\gamma,2(\alpha+\beta)-2}_{T,p}\}}
\end{align*} for $p\in [1,\infty]$, equipped with the norm
\begin{align*}
\norm{u-w}_{\mathcal{D}^{\gamma}_{T,p}}:=\norm{u-w}_{\mathcal{L}_{T,p}^{\gamma',\alpha+\beta-1}}+\norm{u^{\prime}-w^{\prime}}_{(\mathcal{L}_{T,p}^{\gamma,\alpha+\beta-1})^{d}}+\norm{u^{\sharp}-w^{\sharp}}_{\mathcal{L}_{T,p}^{\gamma,2(\alpha+\beta)-1}},
\end{align*} which makes the space a Banach space.\\
For $\mu\in \calC_{p}^{-\epsilon}$, $\epsilon= -((2-\tilde{\gamma})\alpha+2\beta-2)$, we first prove that we obtain a paracontrolled solution $\rho\in \mathcal{D}^{\gamma}_{T,p}$. As the proof is similar to \cite[Theorem 4.7]{kp-sk}, we only give the essential arguments of the proof. Notice that compared to \cite[Theorem 4.7]{kp-sk}, here we consider the operator $\mathfrak{L}^{*}$ instead of $\mathfrak{L}$ and initial conditions in $\calC^{-\epsilon}_{p}$ for $\epsilon=-((2-\tilde{\gamma})\alpha+2\beta-2)$, hence $\rho_{0}=\rho_{0}^{\sharp}$.\\
For $\rho\in \mathcal{D}^{\gamma}_{T,p}$ the resonant product $F\reso\rho =(F^{i}\reso\rho)_{i=1,..,d}$ is well-defined and satisfies
\begin{align*}
F^{i}\reso\rho=F^{i}\reso\rho^{\sharp}+\rho^{\prime}\cdot(F^{i}\reso I_{t}(\nabla\cdot F))+ C_{1}(\rho^{\prime},I_{t}(\nabla\cdot F),F^{i})
\end{align*} for the paraproduct commutator
\begin{align*}
C_{1}(f,g,h):=(f\para g)\reso h-f\cdot(g\reso h).
\end{align*}
Using the paraproduct estimates, we obtain Lipschitz dependence of the product on $(F,\rho)\in\mathcal{X}^{\beta,\gamma'}_{\infty}\times\mathcal{D}^{\gamma}_{T,p}$, that is,
\begin{align*}
\hspace{1em}&\hspace{-1em}
\norm{F\reso\rho}_{\mathcal{M}_{T}^{\gamma'}\calC^{\alpha+2\beta-1}_{p}}
\\&\lesssim \norm{F}_{\mathcal{X}^{\beta,\gamma'}_{\infty}}(1+\norm{F}_{\mathcal{X}^{\beta,\gamma'}_{\infty}})\paren[\big]{\norm{\rho}_{\mathcal{M}_{T}^{\gamma'}\calC^{\alpha+\beta-1}_{p}}+\norm{\rho^{\prime}}_{(\mathcal{M}_{T}^{\gamma'}\calC^{\alpha+\beta-1-\delta}_{p})^{d}}+\norm{\rho^{\sharp}}_{\mathcal{M}_{T}^{\gamma'}\calC^{2(\alpha+\beta)-2-\delta}_{p}}}\\
&\lesssim\norm{F}_{\mathcal{X}^{\beta,\gamma'}_{\infty}}(1+\norm{F}_{\mathcal{X}^{\beta,\gamma'}_{\infty}})\paren[\big]{\norm{\rho}_{\mathcal{M}_{T}^{\gamma'}\calC^{\alpha+\beta-1}_{p}}+\norm{\rho^{\prime}}_{(\mathfrak{L}_{T,p}^{\gamma,\alpha+\beta-1})^{d}}+\norm{\rho^{\sharp}}_{\mathfrak{L}_{T,p}^{\gamma,2(\alpha+\beta)-2}}}\\&\lesssim\norm{F}_{\mathcal{X}^{\beta,\gamma'}_{\infty}}(1+\norm{F}_{\mathcal{X}^{\beta,\gamma'}_{\infty}})\norm{\rho}_{\mathcal{D}_{T,p}^{\gamma}}
\end{align*} for $\delta=\alpha-\alpha\frac{\gamma'}{\gamma}$, using moreover the interpolation estimate from \cite[Lemma 3.7, (3.13)]{kp-sk}.\\
The contraction map will be defined as
\begin{align*}
\mathcal{D}^{\gamma}_{\overline{T},p}\ni (\rho,\rho^{\prime})\mapsto (\phi(\rho),\rho)\in \mathcal{D}^{\gamma}_{\overline{T},p}
\end{align*} with
\begin{align*}
\phi(\rho)_{t}:=P_{t}\mu+I_{t}(-\nabla\cdot (F\rho)).
\end{align*}
Here, $\overline{T}$ will be chosen small enough, such that the above map becomes a contraction. Afterwards the solutions on the subintervals of length $\overline{T}$ are patched together. Notice that the fixed point satisfies $\rho^{\prime}=\rho$.\\
As $\epsilon=-((2-\tilde{\gamma})\alpha+2\beta-2)$, we obtain by the semigroup estimates from \cite[Lemma 2.5]{kp-sk}, that
\begin{align}\label{eq:T-est}
\norm{ P_{t}\mu}_{\calC^{2(\alpha+\beta)-2}_{p}}\lesssim t^{-\tilde{\gamma}}\norm{\mu}_{\calC^{-\epsilon}_{p}}.
\end{align} 
Utilizing the Schauder estimates \cite[Corollary 3.2]{kp-sk} (which apply by a time change also for blow-up-spaces with blow-up at $t=0$ instead of blow-ups at $t=T$) and the estimate for the resonant product yields
\begin{align*}
\norm{I(\nabla\cdot (F\rho))}_{\mathcal{L}_{T,p}^{\gamma,\alpha+\beta-1}}&\lesssim T^{\gamma-\gamma'}\norm{\nabla\cdot (F\rho)}_{\mathcal{M}_{T,0}^{\gamma'}\calC^{\beta-1}_{p}}\\&\lesssim T^{\gamma-\gamma'}\norm{F}_{\mathcal{X}^{\beta,\gamma'}_{\infty}}(1+\norm{F}_{\mathcal{X}^{\beta,\gamma'}_{\infty}})\norm{\rho}_{\mathcal{D}^{\gamma}_{T,p}}.
\end{align*}
Moreover, we have that for a solution $\rho$,
\begin{align*}
\rho_{t}^{\sharp}=P_{t}\mu + C_{2}(\rho,\nabla\cdot F)_{t}+I_{t}(-\nabla\cdot (\rho\reso F))+I_{t}(-\nabla\cdot(\rho\arap F))+I_{t}(-\nabla\rho\para F)
\end{align*} for the semigroup commutator
\begin{align*}
C_{2}(u,v)=I(u\para v)- u\para I(v).
\end{align*} 
Using \eqref{eq:T-est} and \cite[Corollary 3.2]{kp-sk}, we obtain
\begin{align*}
\norm{\rho^{\sharp}}_{\mathcal{L}_{T,p}^{\gamma,2(\alpha+\beta)-2}}\lesssim \norm{\mu}_{\calC^{-\epsilon}_{p}}+T^{\gamma-\gamma'}\norm{F}_{\mathcal{X}^{\beta,\gamma'}_{\infty}}\norm{\rho}_{\mathfrak{L}^{\gamma',\alpha+\beta-1}_{T,p}}.
\end{align*}
Hence, as $\gamma>\gamma'$, replacing $T$ by $\overline{T}\leqslant T$ small enough, we obtain a paracontrolled solution in $\mathcal{D}^{\gamma}_{\overline{T},p}$. Then, we paste the solutions on the subintervals together to obtain a solution on $[0,T]$, cf.~in the proof of \cite[Theorem 4.7]{kp-sk}.\\ 
It remains to justify that the solution at fixed times $t>0$ satisfies $(\rho_{t},\rho_{t}^{\sharp})\in\calC^{\alpha+\beta-1}\times\calC^{2(\alpha+\beta-1)}$, i.e. that we can increase the integrability from $p$ to $\infty$. From the above, we obtain $(\rho,\rho^{\sharp})\in C ([t,T],\calC^{\alpha+\beta-1}_p)\times C([t,T],\calC^{2(\alpha+\beta-1)}_p)$. Then, we can apply the argument to increase the integrability, that was carried out in the end of the proof of  \cite[Proposition 2.4]{PvZ}, to obtain that indeed $(\rho,\rho^{\sharp})\in C ([t,T],\calC^{\alpha+\beta-1})\times C([t,T],\calC^{2(\alpha+\beta-1)})$ for any $t\in (0,T)$.\\
The continuous dependence of the solution on the data $(F,\mu)$ follows analogously as in \cite[Theorem 4.12]{kp-sk}, with the above estimates and a Gronwall-type argument.\\
If $(F,\mu)$ are $1$-periodic distributions, then $P_{t}\mu=p_{t}\ast\mu$ is $1$-periodic, as the convolution with the fractional heat-kernel $p_{t}$ with a periodic distribution yields a periodic function and the fixed point argument can be carried out in the periodic solution space $\mathcal{D}_{T,p}^{\gamma}(\mathbb{T}^{d})$.
\end{proof}
\begin{corollary}\label{cor:density}
Let $X$ be the unique martingale solution of the singular periodic SDE \eqref{eq:sde} for $\mathfrak{L}$ (acting on functions $f:\R^{d}\to\R$), starting at $x\in\R^{d}$. Let $(t,y)\mapsto\rho_{t}(x,y)$ be the mild solution of the Fokker-Planck equation with $\rho_{0}=\delta_{x}$ from \cref{thm:ex-fp-d}. Then for any $t>0$, the map $(x,y)\mapsto \rho_{t}(x,y)$ is continuous.\\
Furthermore, for any $f\in L^{\infty}(\R^{d})$, 
\begin{align}\label{eq:d}
\E_{X_{0}=x}[f(X_{t})]=\int_{\R^{d}} f(y)\rho_{t}(x,y)dy,
\end{align} that is, $\rho_{t}(x,\cdot)$ is the density of $Law(X_{t})$, if $X_{0}=x$, with respect to the Lebesgue measure. In particular, for the projected solution $X^{\mathbb{T}^{d}}$ with drift $F\in\mathcal{X}^{\beta,\gamma'}_{\infty}(\mathbb{T}^{d})$ and $f\in L^{\infty}(\mathbb{T}^{d})$ and $z\in\mathbb{T}^{d}$,
\begin{align}\label{eq:tilde-d}
\E_{X_{0}^{\mathbb{T}^{d}}=z}[f(X_{t}^{\mathbb{T}^{d}})]=\int_{\mathbb{T}^{d}} f(w)\rho_{t}(z,w)dw,
\end{align} where, by abusing notation to not introduce a new symbol for the density on the torus,
$\rho_{t}(z,w):=\rho_{t}(x,y)$ for $(x,y)\in\R^{d}$ with $(\iota(x),\iota(y))=(z,w)$, $\iota:\R^{d}\to\mathbb{T}^{d}$ denoting the canonical projection.
\end{corollary}

\begin{remark}\label{rmk:identity}
Let $\rho(x,\cdot)$ be the solution of the Fokker-Planck equation started in $\delta_{x}$ from \cref{thm:ex-fp-d} and let $u^{y}$ solve the Kolmogorov backward equation with terminal condition $u_{T}=\delta_{y}$ whose existence follows from \cite[Theorem 4.7]{kp-sk}. Then due to \eqref{eq:d} and the Feynman-Kac formula (approximating $F$ and utilizing the continuity of the solutions maps) we see the equality $\rho_{t}(x,y)=u^{y}_{T-t}(x)$.
\end{remark}
\begin{remark}\label{rmk:FP-torus}
If $F\in\mathcal{X}^{\beta,\gamma'}_{\infty}(\mathbb{T}^{d})$, then by definition of $(P_{t}^{\mathbb{T}^{d}})$, $\rho(z,\cdot)$ is the mild solution of the Fokker-Planck equation on the torus (that is, $(P_{t})$ replaced by $(P_{t}^{\mathbb{T}^{d}})$ in \eqref{eq:fp-mild-form}) with $\rho_{0}(z,\cdot)=\delta_{z}$. 
\end{remark}
\begin{proof}
Continuity in $y$ follows from $\rho_{t}(x,\cdot)\in \calC^{\alpha+\beta-1}$ and $\alpha+\beta-1>0$. Continuity in $x$ follows from the continuous dependence of the solution on the initial condition $\delta_{x}$ and continuity of the map $x\mapsto\delta_{x}\in\calC^{-\epsilon}_{1}$ for $\epsilon>0$.\\
That $\rho_{t}$ is the density of $\text{Law}(X_{t})$ follows by approximation of $F$ by $F^{m}\in C^{\infty}_{b}(\R^{d})$ with $F^{m}\to F$ in $\calC^{\beta}_{\R^{d}}$, respectively in $\mathcal{X}^{\beta,\gamma'}_{\infty}$, using that $\rho$ depends continuously on the data $(F,\mu)$ and that $X^{m}\to X$ in distribution, where $X^{m}$ is the strong solution to the SDE with drift term $F^{m}$ (cf.~the proof of \cref{thm:mainthm1}) and the Feynman-Kac formula for classical SDEs. Indeed, for $m\in\N$, we have that for $f\in C^{2}_{b}$ (and thus for $f\in L^{\infty}$ by approximation),
\begin{align*}
u^{m}_{T-t}(x)=\E_{X_{0}^{m}=x}[f(X_{t}^{m})]=\int f(y)\rho_{t}^{m}(x,y)dy
\end{align*} with $(\partial_{t}+\mathfrak{L}^{m})u^{m}=\mathcal{G}^{F^{m}}u^{m}=0$, $u^{m}_{T}=f$, and $(\partial_{t}-(\mathfrak{L}^{m})^{*})\rho=0$, $\rho_{0}=\delta_{x}$. Now, we let $m\to\infty$ to obtain \eqref{eq:d}. In particular, $\rho_{t}\geqslant 0$ and $\rho_{t}\in L^{1}(dx)$. That $\rho_{t}$ is well-defined follows as $\rho_{t}$ is periodic (due to the periodicity assumption on $F$). Equality \eqref{eq:tilde-d} follows from \eqref{eq:d} considering $f\circ \iota$ instead of $f$.
\end{proof}
\begin{proposition}\label{prop:max-p}
Let $\mu\in\calC^{0}_{1}$ be a positive, nontrivial ($\mu\neq 0$) measure. Let $\rho$ be the mild solution of the Fokker-Planck equation $(\partial_{t}-\mathfrak{L}^{\ast})\rho_{t}=0$ with $\rho_{0}=\mu$.  
Then for any compact $K\subset\R^{d}$ and any $t>0$, there exists $c>0$ such that
\begin{align*}
\min_{x\in K}\rho_{t}(x)\geqslant c>0.
\end{align*} 
Let $\rho_{t}$ be as in \cref{rmk:FP-torus}. Then, in particular, for any $z\in\mathbb{T}^{d}$, $t>0$, there exists $c>0$ such that
\begin{align*}
\min_{x\in\mathbb{T}^{d}}\rho_{t}(z,x)\geqslant c>0.
\end{align*} 
\end{proposition}
\begin{proof}
In the Brownian case, $\alpha=2$, this follows from the proof of \cite[Theorem 5.1]{cfg}. We give the adjusted argument for $\alpha\in (1,2]$.\\ Let $p_{t}$ be the $\alpha$-stable density of $L_{t}$. Without loss of generality, we assume $\mu=u\in C_b(\R^{d})$ with $u\geqslant 0$ and with $u\geqslant 1$ on a ball $B(0,\kappa)$, $\kappa>0$. Otherwise, we may consider $\rho_{s}$ for $s>0$ as an initial condition, for which we know that $\rho_{s}\in\calC^{\alpha+\beta-1}\subset C_b(\R^{d})$ and that $\rho_{s}\geqslant 0$ by \cref{cor:density}. Then by continuity there exists a ball $B(x,\kappa)$ where $\rho_{s}>0$. Dividing by the lower bound and shifting $\rho_{s}$, we can assume that $\rho_{s}>1$ on $B(0,\kappa)$.\\ Let now $\kappa>0$ and $u\in C_{b}(\R^{d})$ with $u\geqslant 0$ and with $u\geqslant 1$ on the ball $B(0,\kappa)$. Then by the scaling property, we have that
\begin{align*}
p_{t}\ast u(y)\geqslant \p(\abs{y+t^{1/\alpha}L_{1}}\leqslant\kappa)=\p(L_{1}\in B(yt^{-1/\alpha},\kappa t^{-1/\alpha}))
\end{align*}
Let $y=(\kappa+t\rho)z$ for $z\in B(0,1)$, $\rho\geqslant 0$, so that $y\in B(0,\kappa+t\rho)$. Then we obtain 
\begin{align*}
\MoveEqLeft
\p(L_{1}\in B(yt^{-1/\alpha},\kappa t^{-1/\alpha}))\\&=\p(L_{1}\in B(z(\kappa t^{-1/\alpha}+\rho t^{1-1/\alpha}),\kappa t^{-1/\alpha}))
\\&\geqslant\p(2 z\cdot L_{1}\geqslant \abs{L_{1}}^{2} (\kappa t^{-1/\alpha}+\rho t^{1-1/\alpha})^{-1} +(\abs{z}^{2}-1)[\kappa t^{-1/\alpha}
+\rho t^{1-1/\alpha}])
\\&\geqslant\inf_{\abs{z}\leqslant 1} \p(2 z\cdot L_{1}\geqslant \abs{L_{1}}^{2} (\kappa t^{-1/\alpha}+\rho t^{1-1/\alpha})^{-1} + (\abs{z}^{2}-1)[\kappa t^{-1/\alpha}+\rho t^{1-1/\alpha}])
\allowdisplaybreaks
\\&=\inf_{\abs{z}= 1} \p(2 z\cdot L_{1}\geqslant \abs{L_{1}}^{2} (\kappa t^{-1/\alpha}+\rho t^{1-1/\alpha})^{-1} )
\\&\to \inf_{\abs{z}= 1}\p( z\cdot L_{1}\geqslant 0) =\frac{1}{2}
\end{align*} for $t\to 0$. Here we used that $\alpha>1$ and that by symmetry of $L$, for any $z\in B(0,1)$ with $\abs{z}=1$, $\p(z\cdot L_{1}\geqslant 0)=\p(z\cdot L_{1}\leqslant 0)=1-\p(z\cdot L_{1}\geqslant 0)$, because $\p(z\cdot L_{1}=0)=\p(L_{1}=0)=0$.\\ 
Thus, we conclude, that there exists $t_{\rho}>0$, such that for all $t\in [0,t_{\rho}]$ and all $y\in B(0,\kappa+t\rho)$, $p_{t}\ast u(y)\geqslant \frac{1}{4}$.\\
Moreover, we have 
\begin{align*}
\rho_{t}=P_{t}u+\int_{0}^{t}P_{t-s}(-\nabla\cdot (F\rho_{s}))ds
\end{align*} with $P_{t}u=p_{t}\ast u$ and
\begin{align*}
\norm[\bigg]{\int_{0}^{t}P_{t-s}(-\nabla\cdot (F\rho_{s}))ds}_{L^{\infty}}\leqslant Ct^{(\alpha+\beta-1-\epsilon)/\alpha}
\end{align*} for $\epsilon\in (0,\alpha+\beta-1)$ by the semigroup estimates, \cite[Lemma 2.5]{kp-sk}, with $\alpha+\beta-1>0$. Hence, for small enough $t$, we can achieve
\begin{align*}
\norm[\bigg]{\int_{0}^{t}P_{t-s}(-\nabla\cdot (F\rho_{s}))ds}_{L^{\infty}}< \frac{1}{8}.
\end{align*} 
Together with the lower bound for $p_{t}\ast u$, we obtain that there exists $t_{\rho}>0$, such that for all $t\in [0,t_{\rho}]$ and all $y\in B(0,\kappa+t\rho)$, it holds that
\begin{align*}
\rho_{t}(y)\geqslant \frac{1}{8}.
\end{align*}
Using linearity of the equation, we can repeat that argument on $[t_{\rho},2t_{\rho}]$ etc. Because $K$ is compact, finitely many steps suffice (for large enough $t$, the ball $B(0,\kappa+t\rho)$ will cover $K$) to conclude that for all $T>0$ there exists $c>0$ such that for all $y\in K$ and all $t\in [0,T]$,
\begin{equation*}
\rho_{t}(y)\geqslant c>0.
\qedhere
\end{equation*}
\end{proof}
\end{section}
\begin{section}{Singular resolvent equation}\label{sec:res-eq}
In this and all subsequent sections of this paper, we write $(P_{t})$, respectively $(T_{t})$, for the semigroups acting on the periodic Besov spaces $\calC^{\theta}_{p}(\mathbb{T}^{d})$, $p=2, \infty$, omitting the supercript $\mathbb{T}^{d}$ that we introduced earlier.\\
We solve the resolvent equation in \cref{thm:res-eq} for the singular operator $\mathfrak{L}$ and for singular paracontrolled right-hand sides $G=G^{\sharp}+G^{\prime}\para F$, $G^{\sharp}\in\calC^{0}_{2}(\mathbb{T}^{d})$, $G^{\prime}\in(\calC^{\alpha+\beta-1}_{2}(\mathbb{T}^{d}))^d$, that is
\begin{align*}
(\lambda-\mathfrak{L})g=G,
\end{align*} 
obtaining a solution $g\in\calC^{\alpha+\beta}_{2}(\mathbb{T}^{d})$.\\ 
The next Lemma proves semigroup and commutator estimates for the $I_{\lambda}$-operator.
\begin{lemma}\label{lem:lambda-comm}
Let $\lambda\geqslant 1$, $\delta\in\R$ and $v\in\calC^{\delta}_{2}$. Let again $I_{\lambda}(v):=\int_{0}^{\infty}e^{-\lambda t}P_{t}v dt$. 
Then, $I_{\lambda}(v)$ is well-defined in $\calC^{\beta+\vartheta}_{2}(\mathbb{T}^{d})$ for $\vartheta\in [0,\alpha]$ and the following estimate holds true
\begin{align}\label{eq:a}
\norm{I_{\lambda}(v)}_{\calC^{\delta+\vartheta}_{2}(\mathbb{T}^{d})}\lesssim \lambda^{-(1-\vartheta/\alpha)}\norm{v}_{\calC^{\delta}_{2}(\mathbb{T}^{d})}.
\end{align}
Furthermore, for $v\in\calC^{\sigma}_{2}(\mathbb{T}^{d})$, $\sigma<1$, $u\in\calC^{\beta}(\mathbb{T}^{d})$, $\beta\in\R$, and $\vartheta\in [0,\alpha]$, the following commutator estimate holds true:
\begin{align}\label{eq:b}
\norm{C_{\lambda}(v,u)}_{\calC^{\sigma+\beta+\vartheta}_{2}(\mathbb{T}^{d})}&:=\norm{I_{\lambda}(v\para u)-v\para I_{\lambda}(u)}_{\calC^{\sigma+\beta+\vartheta}_{2}(\mathbb{T}^{d})}\nonumber\\&\lesssim \lambda^{-(1-\vartheta/\alpha)}\norm{v}_{\calC^{\sigma}_{2}(\mathbb{T}^{d})}\norm{u}_{\calC^{\beta}(\mathbb{T}^{d})}.
\end{align}
\begin{proof}
The proof of \eqref{eq:a} follows from the semigroup estimates, \cref{lem:periodic-semi-est}. Indeed, we have 
\begin{align*}
\norm{I_{\lambda}(v)}_{\calC^{\delta+\vartheta}_{2}(\mathbb{T}^{d})}&\leqslant \int_{0}^{\infty} e^{-\lambda t}\norm{P_{t}v}_{\calC^{\delta+\vartheta}_{2}(\mathbb{T}^{d})}dt\\&\lesssim \norm{v}_{\calC^{\vartheta}_{2}(\mathbb{T}^{d})}\int_{0}^{\infty}e^{-\lambda t}[t^{-\vartheta/\alpha}\vee 1]dt\\&= \norm{v}_{\calC^{\vartheta}_{2}(\mathbb{T}^{d})}\paren[\bigg]{\lambda^{-(1-\vartheta/\alpha)}\int_{0}^{1}e^{-t}t^{-\vartheta/\alpha}dt+\lambda^{-1}\int_{1}^{\infty}e^{-t}dt}\\&\lesssim  \lambda^{-(1-\vartheta/\alpha)}\norm{v}_{\calC^{\vartheta}_{2}(\mathbb{T}^{d})},
\end{align*} since $\lambda\geqslant 1$ and where we use that $\int_{0}^{1}e^{-t}t^{-\vartheta/\alpha}dt\leqslant \int_{0}^{1}t^{-\vartheta/\alpha}dt<\infty$ if $\vartheta\in [0,\alpha)$ and $\int_{1}^{\infty} e^{-t}dt<\infty$. The bound in the case $\vartheta=\alpha$ follows with 
\begin{align*}
\norm{I_{\lambda}(v)}_{\calC^{\delta+\alpha}_{p}(\mathbb{T}^{d})}&\leqslant \norm[\bigg]{\int_{0}^{1} e^{-\lambda t}P_{t}vdt}_{\calC^{\delta+\alpha}_{2}(\mathbb{T}^{d})}+\int_{1}^{\infty}e^{-\lambda t}\norm{P_{t}v}_{\calC^{\delta+\alpha}_{p}(\mathbb{T}^{d})}dt
\\&\lesssim\norm{v}_{\calC^{\delta}_{p}(\mathbb{T}^{d})},
\end{align*} using \cite[Lemma 3.1]{kp-sk} to estimate the integral over $[0,1]$ (with, in the notation of that lemma, $T=1$, $\gamma=0$, $\sigma=\delta$, $\varsigma=\alpha$, $f_{0,t}=e^{-\lambda t}P_{t}v$).\\
The commutator \eqref{eq:b} is proven analogously using \cite[Lemma 2.7]{kp-sk}. 
\end{proof}

\end{lemma}
\begin{theorem}\label{thm:res-eq}
Let $\alpha\in (1,2]$ and $F\in\calC^{\beta}(\mathbb{T}^{d})$ for $\beta\in (\frac{1-\alpha}{2},0)$ or $F\in\mathcal{X}^{\beta,\gamma}_{\infty}(\mathbb{T}^{d})$ for $\beta\in (\frac{2-2\alpha}{3},\frac{1-\alpha}{2}]$ and $\gamma\in (\frac{2\beta+2\alpha-1}{\alpha},1)$.\\ 
Then, for $\lambda>0$ large enough, the resolvent equation
\begin{align}
R_{\lambda}g=(\lambda-\mathfrak{L})g=G
\end{align} with right-hand side $G=G^{\sharp}+G^{\prime}\para F$, $G^{\sharp}\in \calC^{0}_{2}(\mathbb{T}^{d})$, $G^{\prime}\in(\calC^{\alpha+\beta-1}_{2}(\mathbb{T}^{d}))^{d}$, possesses a unique solution $g\in\calC^{\theta}_{2}(\mathbb{T}^{d})$, $\theta\in ((2-\beta)/2,\beta+\alpha)$.\\ If $\beta\in (\frac{2-2\alpha}{3},\frac{1-\alpha}{2}]$, the solution is paracontrolled, that is,
\begin{align}
g=g^{\sharp}+(G^{\prime}+\nabla g)\para I_{\lambda}(F),\quad g^{\sharp}\in\calC^{2\theta-1}_{2}(\mathbb{T}^{d}).
\end{align}  
\end{theorem}
\begin{proof}
Consider the paracontrolled solution space
\begin{align}
\mathcal{D}^{\theta}_{2}:=\{(g,g^{\prime})\in \calC^{\theta}_{2}(\mathbb{T}^{d})\times (\calC^{\theta-1}_{2}(\mathbb{T}^{d}))^{d}\mid g^{\sharp}:=g-g^{\prime}\para I_{\lambda}(F)\in\calC^{2\theta-1}_{2}(\mathbb{T}^{d})\}
\end{align} with norm $\norm{g-h}_{\mathcal{D}^{\theta}_{2}}:=\norm{g-h}_{\calC^{\theta}_{2}(\mathbb{T}^{d})}+\norm{g^{\sharp}-h^{\sharp}}_{\calC^{2\theta-1}_{2}(\mathbb{T}^{d})}+\norm{g^{\prime}-h^{\prime}}_{\calC^{\theta-1}_{2}(\mathbb{T}^{d})}$, which makes it a Banach space.\\
The solution $g$ satisfies
\begin{align*}
g=\int_{0}^{\infty}e^{-\lambda t} P_{t}(G+F\cdot\nabla g)dt,
\end{align*} i.e. it is the fixed point of the map $\calC^{\theta}_{2}(\mathbb{T}^{d})\ni g\mapsto \phi_{\lambda}(g):=\int_{0}^{\infty}e^{-\lambda t} P_{t}(G+F\cdot\nabla g)dt\in \calC^{\theta}_{2}(\mathbb{T}^{d})$, respectively, in the rough case $\beta\in ((2-2\alpha)/3,(1-\alpha)/2]$, of the map 
\begin{align*}
\mathcal{D}^{\theta}_{2}\ni (g,g')\mapsto (\phi_{\lambda}(g),G^{\prime}+\nabla g)=:\Phi_{\lambda}(g,g^{\prime})\in \mathcal{D}^{\theta}_{2}.
\end{align*}
The product is defined as $F\cdot\nabla g:=F\reso\nabla g+F\para\nabla g+F\arap\nabla g$, where for $F\in\mathcal{X}^{\beta,\gamma}_{\infty}(\mathbb{T}^{d})$ and $g\in\mathcal{D}^{\theta}_{2}$,
\begin{align*}
F\reso\nabla g=\sum_{i=1}^{d}F^{i}\reso\partial_{i} g&:=\sum_{i=1}^{d}\Big[F^{i}\reso [\partial_{i} g^{\sharp}+\partial_{i} g^{\prime}\para I_{\lambda}(F)]+g(I_{\lambda}(\partial_{i} F)\reso F^{i})\\&\qquad\qquad+C_{1}(g,I_{\lambda}(\partial_{i} F),F^{i})\Big],
\end{align*} with paraproduct commutator 
\begin{align}\label{eq:para-comm}
C_{1}(g,f,h):=(g\para f)\reso h-g(f\reso h)
\end{align} from \cite[Lemma 2.4]{Gubinelli2015Paracontrolled}. Analogously as before, the product of $F\in\mathcal{X}^{\beta,\gamma}_{\infty}(\mathbb{T}^{d})$ and $g\in\mathcal{D}^{\theta}_{2}$ with $\theta>(2-\beta)/2$ can thus be estimated by 
\begin{align*}
\norm{F\cdot g}_{\calC^{\beta}_{2}(\mathbb{T}^{d})}\lesssim\norm{F}_{\mathcal{X}^{\beta,\gamma}_{\infty}(\mathbb{T}^{d})}(1+\norm{F}_{\mathcal{X}^{\beta,\gamma}_{\infty}(\mathbb{T}^{d})})\norm{g}_{\mathcal{D}^{\theta}_{2}}.
\end{align*} 
The unique fixed point is obtained by the Banach fixed point theorem, where, in the Young case the map $\phi$, and in the rough case, $\Phi^{2}_{\lambda}=\Phi_{\lambda}\circ\Phi_{\lambda}$ are contractions for large enough $\lambda>0$. This can be seen by estimating
\begin{align*}
\norm{\phi_{\lambda}(g)-\phi_{\lambda}(h)}_{\calC^{\theta}_{2}(\mathbb{T}^{d})}&\lesssim\lambda^{(\theta-\beta-\alpha)/\alpha}\norm{F\cdot\nabla (g-h)}_{\calC^{\beta}_{2}(\mathbb{T}^{d})}\\&\lesssim\lambda^{(\theta-\beta-\alpha)/\alpha}\norm{F}_{\mathcal{X}^{\beta,\gamma}_{\infty}(\mathbb{T}^{d})}(1+\norm{F}_{\mathcal{X}^{\beta,\gamma}_{\infty}(\mathbb{T}^{d})})\norm{g-h}_{\mathcal{D}^{\theta}_{2}}
\end{align*} using \eqref{eq:a} and the estimate for the product. Thus a contraction is obtained by choosing $\lambda$ large enough, such that $\lambda^{(\theta-\beta-\alpha)/\alpha}\norm{F}_{\mathcal{X}^{\beta,\gamma}_{\infty}(\mathbb{T}^{d})}(1+\norm{F}_{\mathcal{X}^{\beta,\gamma}_{\infty}(\mathbb{T}^{d})})<1$, using $\theta<\alpha+\beta$. To check that indeed $\Phi_{\lambda}(g,g^{\prime})\in \mathcal{D}^{\theta}_{2}$, we note that 
\begin{align*}
\Phi_{\lambda}(g,g^{\prime})^{\sharp}&=\phi_{\lambda}(g)-[G^{\prime}e_{i}+\nabla g]\para I_{\lambda}(F)\\&=I_{\lambda}(G^{\sharp}+F\reso g+F\para\nabla g)+C_{\lambda}(G^{\prime}e_{i}+\nabla g, F)
\end{align*} for the commutator $C_{\lambda}$ from \eqref{eq:b}. Notice that, if $\beta<(1-\alpha)/2$, for $G^{\sharp}\in\calC^{0}_{2}(\mathbb{T}^{d})$, $I_{\lambda}(G^{\sharp})\in\calC^{\alpha}_{2}(\mathbb{T}^{d})\subset\calC^{2\theta-1}_{2}(\mathbb{T}^{d}) $ as $\theta>(1+\alpha)/2$. Hence, together with \cref{lem:lambda-comm}, it follows that $\Phi_{\lambda}(g,g^{\prime})^{\sharp}\in\calC^{2\theta-1}(\mathbb{T}^{d})$. Thereby we also get the small factor of $\lambda^{(\theta-\alpha-\beta)/\alpha}$ in the estimate.  To see that $\Phi^{2}_{\lambda}=\Phi_{\lambda}\circ\Phi_{\lambda}$ is a contraction, we furthermore check  
\begin{align*}
\MoveEqLeft
\norm{\Phi_{\lambda}(\Phi_{\lambda}(g,g'))'-\Phi_{\lambda}(\Phi_{\lambda}(h,h'))'}_{\calC^{\theta-1}_{2}(\mathbb{T}^{d})}\\&= \norm{\nabla\phi_{\lambda}(g)-\nabla\phi_{\lambda}(h)}_{\calC^{\theta-1}_{2}(\mathbb{T}^{d})}\\&\lesssim \norm{\phi_{\lambda}(g)-\phi_{\lambda}(h)}_{\calC^{\theta}_{2}(\mathbb{T}^{d})}\\&\lesssim\lambda^{(\theta-\beta-\alpha)/\alpha}\norm{F}_{\mathcal{X}^{\beta,\gamma}_{\infty}(\mathbb{T}^{d})}(1+\norm{F}_{\mathcal{X}^{\beta,\gamma}_{\infty}(\mathbb{T}^{d})})\norm{g-h}_{\mathcal{D}^{\theta}_{2}},
\end{align*} by the above estimate.
\end{proof}

\end{section}
\begin{section}{Existence of an invariant measure and spectral gap estimates}\label{sec:inv-m}
In this section, we prove with \cref{thm:inv-m} existence and uniqueness of an invariant, ergodic probability measure for the process $X^{\mathbb{T}^{d}}$ with state space $\mathbb{T}^{d}$, in the following for short denoted by $X$. The theorem moreover shows that $X$ is exponentially ergodic, in the sense that pointwise spectral gap estimates for its semigroup $(T_{t})$ hold. Furthermore, we characterize the domain of $\mathfrak{L}$ in $L^{2}(\pi)$ in \cref{thm:generator} and define the mean of $F\in\mathcal{X}^{\beta,\gamma}_{\infty}$ with respect to the invariant measure $\pi$ in \cref{thm:def-F-pi-int}.\\
Existence and uniqueness of the invariant measure together with the pointwise spectral gap estimates on the semigroup are obtained by an application of Doeblin's theorem (see e.g.~\cite[Theorem 3.1, Chapter 3, Section 3, p.~365]{Bensoussan1978}), that we state here in the continuous time setting.
\begin{lemma}[Doeblin's theorem]\label{thm:doeblin}
Let $(X_{t})_{t\geqslant 0}$ be a time-homogeneous Markov process with state space $(S,\Sigma)$ for a compact metric space $S$ and its Borel-sigma-field $\Sigma$. Let $(T_{t})_{t\geqslant 0}$ be the associated semigroup, $T_{t}f(x):=\E[f(X_{t})\mid X_{0}=x]$ for $x\in S$ and $f:S\to\R$ bounded measurable. 
Assume further, that there exists a probability measure $\mu$ on $(S,\Sigma)$ and, for any $t>0$, a continuous function $\rho_{t}:S\times S\to\R^{+}$, such that $T_{t}\mathbf{1}_{E}(x)=\int_{E}\rho_{t}(x,y)\mu(dy)$, $E\in\Sigma$. Assume moreover that, for any $t>0$, there exists an open ball $U_{0}$, such that $\mu(U_{0})>0$ and $\rho_{t}(x,y)>0$ for all $x\in S$ and $y\in U_{0}$.\\ Then, there exists a unique invariant probability measure $\pi$ (i.e. $\int_{S}T_{t}\mathbf{1}_{E}(x)\pi(dx)=\pi(E)$ for all $E\in\Sigma$ and all $t\geqslant 0$) on $(S,\Sigma)$ with the property that there exist constants $K,\nu>0$, such that for all $t\geqslant 0$, $x\in S$ and $\phi:S\to\R$ bounded measurable,  
\begin{align}\label{eq:sg}
\abs[\bigg]{T_{t}\phi(x)-\int_{S}\phi(y)\pi(dy)}\leqslant K\abs{\phi}e^{-\nu t}
\end{align} where $\abs{\phi}:=\sup_{x\in S}\abs{\phi(x)}$.
\end{lemma}
\begin{proof}
For discrete time Markov chains, the result follows immediately from \cite[Theorem 3.1, p. 365]{Bensoussan1978}. For continuous time Markov processes, the proof is similar. Indeed, in the same manner one proves that if $\pi$ is such that \eqref{eq:sg} holds, then $\pi$ is unique and $\pi$ is invariant for $(T_{t})$. Furthermore, using the assumptions on the density $\rho$ and the same proof steps as in \cite[Theorem 3.1, p. 365]{Bensoussan1978}, one obtains existence of an invariant measure $\pi$ with $\pi(E)$ given as the limit of $(T_{n}\mathbf{1}_{E}(x))_{n}$ for any $x\in S$ and with \eqref{eq:sg} for $t$ replaced by $n\in\N$. Then, using the semigroup property, we also obtain \eqref{eq:sg} for any $t\geqslant 0$, with a possibly different constant $K>0$. Indeed, let $t>0$ and $n=\lfloor t\rfloor$. Then for bounded measurable $\phi$ with $\int_{S}\phi d\pi=0$, we obtain
\begin{align*}
\abs{T_{t}\phi(x)}=\abs{T_{n}T_{t-n}\phi(x)}\leqslant K \abs{T_{t-n}\phi}e^{-\nu n}\leqslant K\abs{\phi}e^{-\nu n}=K e^{\nu (t-n)}\abs{\phi}e^{-\nu t}\leqslant Ke^{\nu}\abs{\phi} e^{-\nu t}.
\end{align*} 
Now, by changing the constant $K$, we obtain \eqref{eq:sg} for all $t\geqslant 0$.
\end{proof}
\begin{theorem}\label{thm:inv-m}
Let $X$ be the martingale solution to the singular periodic SDE \eqref{eq:sde} projected onto $\mathbb{T}^{d}$ with contraction semigroup $(T_{t})_{t\geqslant 0}$ on bounded measurable functions $f:\mathbb{T}^{d}\to\R$.\\ 
Then there exists a unique invariant probability measure $\pi$ for $(T_{t})$. In particular, $\pi$ is ergodic for $X$. Furthermore there exist constants $K,\mu>0$ such that for all $f\in L^{\infty}(\mathbb{T}^{d})$,
\begin{align}\label{eq:p-sp-est}
\norm{T_{t}f-\langle f\rangle_{\pi}}_{L^{\infty}}\leqslant  K \norm{f}_{L^{\infty}} e^{-\mu t}.
\end{align} That is, $L^{\infty}$-spectral gap estimates for the associated Markov semigroup $(T_{t})$ hold true. In particular, $\pi$ is absolutely continuous with respect to the Lebesgue measure on the torus, with density denoted by $\rho_{\infty}$.
\end{theorem}
\begin{proof}
The proof is an application of Doeblin's theorem.
We check, that the assumptions of \cref{thm:doeblin} are satisfied. To that aim, note that for the Fokker-Planck density $\rho_{t}(x,\cdot)$ with $\rho_{0}=\delta_{x}$, the map $(x,y)\mapsto \rho_{t}(x,y)$ is continuous by \cref{thm:ex-fp-d}. It remains to show that there exists an open ball $U_{0}$ and a constant $c>0$, such that $\rho_{t}$ is bounded from below by $c$ on $\mathbb{T}^{d}\times U_{0}$. We choose $U_{0}=\mathbb{T}^{d}$ and obtain 
\begin{align*}
\min_{x\in\mathbb{T}^{d},y\in U_{0}}\rho_{t}(x,y)=\rho_{t}(x^*,y^*)\geqslant c>0.
\end{align*} 
Indeed, this follows from the strict maximum principle for $y\mapsto\rho_{t}(x^*,y)$ by \cref{prop:max-p} with $c=c(x^*)>0$.\\ 
The spectral gap estimates also imply absolute continuity, as 
\begin{align*}
\langle 1_{A}\rangle_{\pi}=\lim_{t\to\infty} \E_{X_{0}=x}[1_{A}(X_{t})]=\lim_{t\to\infty} \int 1_{A}(y)\rho_{t}(x,y)dy
\end{align*} and thus any Lebesgue nullset $A$ is also a $\pi$-nullset. The existence of the density thus follows by the Radon-Nikodym theorem. 
\end{proof}

\begin{corollary}\label{cor:inv-d}
Let $\rho_{\infty}$ be the Lebesgue density of the invariant measure $\pi$. Then $\rho_{\infty}\in \calC^{\alpha+\beta-1}(\mathbb{T}^{d})$ and it follows the paracontrolled structure 
\begin{align*}
\rho_{\infty}=\rho_{\infty}^{\sharp}+\rho_{\infty}\para I_{\infty}(\nabla\cdot F),
\end{align*} where $\rho_{\infty}^{\sharp}\in\calC^{2(\alpha+\beta)-2}(\mathbb{T}^{d})$ and $I_{\infty}(\nabla\cdot F):=\int_{0}^{\infty}P_{s}(\nabla\cdot F)ds$.\\ Furthermore, the density is strictly positive,
\begin{align*}
\min_{x\in\mathbb{T}^{d}}\rho_{\infty}(x)>0.
\end{align*}
In particular, $\pi$ is equivalent to the Lebesgue measure. 
\end{corollary}
\begin{proof}
Let $t>0$. By invariance of $\pi$, i.e. $\langle T_{t}f\rangle_{\pi}=\langle f\rangle_{\pi}$ for all $f\in L^{\infty}(\mathbb{T}^{d})$, and $d\pi=\rho_{\infty}dx$, we obtain that almost surely
\begin{align*}
\rho_{\infty}=T_{t}^{\ast}\rho_{\infty},
\end{align*} where $T_{t}^{\ast}$ denotes the adjoint of $T_{t}$ with respect to $L^{2}(\lambda)$. Here $\lambda$ denotes the Lebesgue measure and $\langle f\rangle_{\pi}:=\int_{\mathbb{T}^{d}}f(x)\pi(dx)$.\\ Denote $y_{t}(x):=T_{t}^{\ast}\rho_{\infty} (x)$. Then we show that $y$ is a mild solution of the Fokker-Planck equation started in $\rho_{\infty}$, that is 
\begin{align}\label{eq:fp-y}
(\partial_{t}-\mathfrak{L}^{\ast})y=0,\quad y_{0}=\rho_{\infty}.
\end{align} 
Here the density satisfies $\rho_{\infty}\in L^{1}(\lambda)$, i.p. $\rho_{\infty}\in \calC^{0}_{1}(\mathbb{T}^{d})$. Indeed, that $y_{t}=T_{t}^{\ast}\rho_{\infty}$ is a mild solution of the Fokker-Planck equation follows from approximation of $F$ by $F^{m}\in C^{\infty}(\mathbb{T}^{d})$ with $F^{m}\to F$ in $\mathcal{X}^{\beta,\gamma}_{\infty}(\mathbb{T}^{d})$ using that for $m\in\N$, $y^{m}=(T^{m}_t)^{\ast}\rho_{\infty}$ solves $(\partial_{t}-(\mathfrak{L}^{m})^{\ast})y^{m}=0$, $y^{m}=\rho_{\infty}$ by the classical Fokker-Planck theory, where $T^{m}$ denotes the semigroup for the strong solution of the SDE with drift $F^{m}$ and generator $\mathfrak{L}^{m}:=-\La+F^{m}\cdot\nabla$. By continuity of the Fokker-Planck solution map from \cref{thm:ex-fp-d} for converging data $(F^{m},\rho_{\infty})\to (F,\rho_{\infty})$ in $\mathcal{X}^{\beta,\gamma}_{\infty}(\mathbb{T}^{d})\times\calC^{0}_{1}(\mathbb{T}^{d})$, we deduce $y^{m}\to y$ in the paracontrolled solution space, where $y$ is the mild solution of \eqref{eq:fp-y}.\\
The lower bound away from zero then also follows from \cref{thm:ex-fp-d}, as well as the paracontrolled structure
\begin{align*}
\rho_{\infty}=\rho_{\infty}^{\sharp}+\rho_{\infty}\para I_{t}(\nabla\cdot F),
\end{align*} where $\rho_{\infty}^{\sharp}:=y_{t}^{\sharp}\in\calC^{2(\alpha+\beta)-2}(\mathbb{T}^{d})$ and $I_{t}(\nabla\cdot F):=\int_{0}^{t}P_{t-s}(\nabla\cdot F)ds$.\\ Due to $\F_{\mathbb{T}^{d}}(\nabla\cdot F)(0)=0$, we have that, for any $\theta\geqslant 0$, there exists $c>0$, such that, uniformly in $s>0$,
\begin{align}\label{eq:schauder-zero-fourier}
\norm{P_{s}(\nabla\cdot F)}_{\calC^{\beta-1+\theta}(\mathbb{T}^{d})}\lesssim s^{-\theta/\alpha}e^{-cs}\norm{\nabla\cdot F}_{\calC^{\beta-1}(\mathbb{T}^{d})}.
\end{align} 
Indeed, this follows from \cref{lem:exp-schauder}.
Thus we obtain, for $t>0$ and any $\theta\geqslant 0$, that
\begin{align*}
I_{t}(\nabla\cdot F)-I_{\infty}(\nabla\cdot F)=\int_{t}^{\infty}P_{s}(\nabla\cdot F)ds\in\calC^{\theta}(\mathbb{T}^{d}).
\end{align*}
That is, the remainder is smooth and thus can be absorbed into $\rho_{\infty}^{\sharp}$. Notice, that $\int_{0}^{t}P_{s}(\nabla\cdot F)ds\in\calC^{\alpha+\beta-1}(\mathbb{T}^{d})$ by \eqref{eq:schauder-zero-fourier} and in particular that $I_{\infty}(\nabla\cdot F)\in\calC^{\alpha+\beta-1}(\mathbb{T}^{d})$ is well-defined. 
\end{proof}

\begin{corollary}\label{cor:L2-spectral-gap}
Let $X$ and $(T_{t})$ be as before. Then, the semigroup $(T_{t})_{t\geqslant 0}$ can be uniquely extended to a strongly continuous contraction semigroup on $L^{2}(\pi)$, i.e. $T_{t+s}=T_{t}T_{s}$, $T_{t}1=1$, $T_{t}f\to f$ for $t\downarrow 0$ and $f\in L^{2}(\pi)$ and $\norm{T_{t}f}_{L^{2}(\pi)}\leqslant\norm{f}_{L^{2}(\pi)}$, such that and for (possibly different) constants $K,\mu>0$, the $L^{2}(\pi)$-spectral gap estimates hold true:
\begin{align*}
\norm{T_{t}f-\langle f\rangle_{\pi}}_{L^{2}(\pi)}\leqslant  K \norm{f}_{L^{2}(\pi)} e^{-\mu t}\quad \text{ for all } f\in L^{2}(\pi).
\end{align*}
\end{corollary}
\begin{proof}
That the semigroup $(T_{t})_{t\geqslant 0}$ can be uniquely extended to a contraction semigroup on $L^{2}(\pi)$ follows from Jensen's inequality, 
\begin{align*}
\norm{T_{t}f}_{L^{2}(\pi)}^{2}=\int \abs{\E_{X_{0}=x}[f(X_{t})]}^{2}\pi(dx)\leqslant\int \E_{X_{0}=x}[\abs{f(X_{t})}^{2}]\pi(dx)=\norm{f}_{L^{2}(\pi)}^{2},
\end{align*}
for $f\in L^{\infty}$, using the invariance of $\pi$ (by \cref{thm:inv-m}).
By approximation, we then also obtain for the extension, that $T_{t}f(x)=\E_{X_{0}=x}[f(X_{t})]$ for $f\in L^{2}(\pi)$.\\ We check strong continuity of the semigroup on $L^{2}(\pi)$. 
Using the contraction property in $L^{2}(\pi)$, we obtain
\begin{align}\label{eq:Test}
\norm{T_{t}f-f}_{L^{2}(\pi)}^{2}=\norm{T_{t}f}_{L^{2}(\pi)}^{2}+\norm{f}_{L^{2}(\pi)}^{2}-2\langle T_{t}f,f\rangle_{\pi}\leqslant 2\norm{f}_{L^{2}(\pi)}^{2}-2\langle T_{t}f,f\rangle_{\pi}.
\end{align}
It is left to prove that the right-hand side vanishes as $t\downarrow 0$.
By Fatou's lemma and using that $X$ is almost surely càdlàg, we have that for $x\in\mathbb{T}^{d}$ and $f\in C(\mathbb{T}^{d},\R)$,
\begin{align}\label{eq:conv}
\lim_{t\downarrow 0}\abs{T_{t}f(x)-f(x)}\leqslant\E_{X_{0}=x}[\lim_{t\downarrow 0}\abs{f(X_{t})-f(X_{0})}]=0.
\end{align}
Furthermore, we can bound uniformly in $x\in\mathbb{T}^{d}$ and $t>0$,
\begin{align}\label{eq:bound-PH}
\abs{T_{t}f(x)}=\abs[\bigg]{\int\rho_{t}(x,y)f(y)dy}\leqslant \frac{\sup_{t>0}\max_{x,y\in\mathbb{T}^{d}}\rho_{t}(x,y)}{\min_{y\in\mathbb{T}^{d}}\rho_{\infty}(y)}\norm{f}_{L^{1}(\pi)}\leqslant C \norm{f}_{L^{2}(\pi)}
\end{align}
where $C>0$ is a constant (not depending on $t$, $f$) and $\rho_{t}(x,y)$ denotes the Fokker-Planck density with $\rho_{0}(x,y)=\delta_{x}$. Here, we have $\min_{y\in\mathbb{T}^{d}}\rho_{\infty}(y)>0$ by \cref{cor:inv-d}. Furthermore we have
\begin{align}\label{eq:unif-bound}
\sup_{t>0}\max_{x,y\in\mathbb{T}^{d}}\rho_{t}(x,y)<\infty.
\end{align} 
Indeed, by the $L^{\infty}$-spectral gap estimates, it follows that 
\begin{align*}
\sup_{t\geqslant 0}\norm{\rho_{t}\ast f}_{L^{\infty}}\leqslant K\norm{f}_{L^{\infty}}+\abs{\langle f\rangle_{\pi}},
\end{align*} with convolution $(\rho_{t}\ast f)(x):=\int_{\mathbb{T}^{d}}\rho_{t}(x,y)f(y)dy$. We can apply this bound for $f^{\epsilon,\tilde{y}}(y):=\mathbf{1}_{\abs{y-\tilde{y}}<\epsilon}$ for $\tilde{y}\in\mathbb{T}^{d}$ and $\epsilon>0$ and let $\epsilon\downarrow 0$. By continuity of $y\mapsto\rho_{t}(x,y)$ and the dominated convergence theorem, $(\rho_{t}\ast f^{\epsilon,\tilde{y}})(x)\to\rho_{t}(x,\tilde{y})\lambda(\mathbb{T}^{d})$, which yields \eqref{eq:unif-bound}.\\
In particular, by \eqref{eq:bound-PH}, $\sup_{t>0}\norm{T_{t}f}_{L^{\infty}}\lesssim \norm{f}_{L^{2}(\pi)}$ and an application of the dominated convergence theorem using \eqref{eq:conv}, yields that for $f\in C(\mathbb{T}^{d},\R)$,
\begin{align*}
\lim_{t\downarrow 0}\langle T_{t}f,f\rangle_{\pi}=\norm{f}_{L^{2}(\pi)}^{2}.
\end{align*}
We conclude with \eqref{eq:Test}, that for all $f\in C(\mathbb{T}^{d},\R)$, $\norm{T_{t}f-f}_{L^{2}(\pi)}\to 0$ as $t\downarrow 0$.\\
As $(T_{t})$ is a contraction semigroup on $L^{2}(\pi)$, the operator norm is trivially bounded, that is $\sup_{t\geqslant 0}\norm{T_{t}}_{L(L^{2}(\pi))}\leqslant 1$. Above, we proved that $(T_{t})$ is strongly continuous on a dense subset of $L^{2}(\pi)$. Thus, together with boundedness of the operator norm, $(T_{t})$ is also strongly continuous on $L^{2}(\pi)$ as a consequence of the Banach-Steinhaus theorem.\\
It remains to prove that the $L^{2}(\pi)$-spectral gap estimates follow from the $L^{\infty}$-spectral gap estimates and the bound \eqref{eq:bound-PH}.
Indeed, we obtain for $f\in L^{2}(\pi)$ with $\langle f\rangle_{\pi}=0$ and all $t>1$,
\begin{align*}
\norm{T_{t}f}_{L^{2}(\pi)}=\norm{T_{t-1}T_{1}
f}_{L^{2}(\pi)}\leqslant K e^{-\mu(t-1)}\norm{T_{1}f}_{L^{\infty}}\leqslant e^{\mu}CK e^{-\mu t}\norm{f}_{L^{2}(\pi)}.
\end{align*} 
For $t\in[0,1]$, we trivially estimate, using the contraction property, 
\begin{align*}
 \norm{T_{t}f}_{L^{2}(\pi)}\leqslant \norm{f}_{L^{2}(\pi)}\leqslant e^{\mu}e^{-\mu t} \norm{f}_{L^{2}(\pi)}.
\end{align*} 
\end{proof}
\begin{remark}
The argument in the above proof of \cref{cor:L2-spectral-gap} (using the bound \eqref{eq:unif-bound} and $\rho_{\infty}>0$) can be adapted to prove the stronger estimate (for constants $K,\mu>0$)
\begin{align*}
\norm{T_{t}f-\langle f\rangle_{\pi}}_{L^{\infty}}\leqslant Ke^{-\mu t}\norm{f-\langle f\rangle_{\pi}}_{L^{1}(\pi)},
\end{align*} which in particular implies the $L^{2}(\pi)$-$L^{2}(\pi)$-bound from the corollary.
\end{remark}
\begin{remark}
More generally, one can show the Feller property, that is $(T_{t})$ is strongly continuous on $C(\mathbb{T}^{d})$. Using \cite[Proposition III.2.4]{Revuz1999} and \eqref{eq:conv}, it is left to show $T_{t}f\in C(\mathbb{T}^{d})$ for $f\in C(\mathbb{T}^{d})\subset\calC^{0}(\mathbb{T}^{d})$. But this follows from \cite[Theorem 4.7]{kp-sk}, since for $R>t$, $y_{t}=T_{R-t}f$ solves the backward Kolmogorov equation with periodic terminal condition $y_{R}=f\in \calC^{0}$ and $y\in\mathcal{M}_{R}^{\gamma}\calC^{\alpha+\beta}$, such that in particular $x\mapsto y_{t}(x)$ is continuous.
\end{remark}
\noindent The next theorem relates the semigroup $(T_{t})_{t\geqslant 0}$ from above with the generator $\mathfrak{L}$ and gives an explicit representation of its domain in terms of paracontrolled solutions of singular resolvent equations.
\begin{theorem}\label{thm:generator}
Let $(T_{t})$ be the contraction semigroup on $L^{2}(\pi)$ from \cref{cor:L2-spectral-gap} and denote its generator by $(A,dom(A))$ with $A:dom(A)\subset L^{2}(\pi)\to L^{2}(\pi)$ and domain $dom(A):=\{f\in L^{2}(\pi)\mid \lim_{t\to 0} (T_{t}f-f)/t =:Af\text{ exists in } L^{2}(\pi)\}$. Let $\theta\in ((1+\alpha)/2,\alpha+\beta)$ and 
\begin{align*}
D:=\{g\in \mathcal{D}^{\theta}_{2}\mid R_{\lambda}g= G\text{ for some } G\in L^{2}(\pi)\text{ and }\lambda>0 \},
\end{align*}
where $R_{\lambda}:=(\lambda-\mathfrak{L})$.\\ 
Then it follows $D = dom(A)$ and $(A,D)=(\mathfrak{L},D)$. In particular, $(\mathfrak{L},D)$ is the generator of the Markov process $X$ with state space $\mathbb{T}^{d}$ and transition semigroup $(T_{t})$.
\end{theorem}
\begin{remark}
Since the drift $F$ does not depend on a time variable, one could reformulate the martingale problem for $X$ in terms of the elliptic generator $\mathfrak{L}$ and the domain $D\subset L^{2}(\pi)$.
\end{remark}
\begin{proof}
We first show that $D\subset dom(A)$. To this aim, note that for $f\in D$, we obtain $R_{\lambda}f=G$ for $G\in L^{2}(\pi)$. For a mollification $(G^{n})\subset C^{\infty}(\mathbb{T}^{d})$ of $G$ and $(f^{n})\subset C^{\infty}(\mathbb{T}^{d})$, such that $R_{\lambda}f^{n}=G^{n}$, we obtain that in particular $f^{n}$ is a mild solution of the Kolmogorov backward equation on the torus for $\mathcal{G}=\partial_{t}+\mathfrak{L}$ with right-hand side $\lambda f^{n}-G^{n}\in L^{\infty}$  and terminal condition $f^{n}\in\calC^{3}$. Equivalently, its periodic version is the periodic solution of the Kolmogorov backward equation on $\R^{d}$.
As $X$ equals the projected solution of the $(\mathcal{G},x)$-martingale problem onto the torus, we have, for $n\in\N$ and $x\in\mathbb{T}^{d}$, that
\begin{align*}
T_{t}f^{n}(x)-f^{n}(x)&=\E_{X_{0}=x}[f^{n}(X_{t})-f^{n}(X_{0})]\\&=\E_{X_{0}=x}\bigg[\int_{0}^{t}(\lambda f^{n}-G^{n})(X_{s})ds\bigg]=\int_{0}^{t}T_{s}(\lambda f^{n}-G^{n})(x)ds.
\end{align*}
Using that $f^{n}\to f$ in $L^{2}(\pi)$ as $G^{n}\to G$ by continuity of the resolvent solution map, we obtain that for $f\in D$,
\begin{align*}
T_{t}f-f=\int_{0}^{t}T_{s}(\lambda f-G)ds.
\end{align*}
By continuity of the map $s\mapsto T_{s}(\lambda f-G)\in L^{2}(\pi)$, since  $T$ is strongly continuous on $L^{2}(\pi)$, we obtain that for $f\in D$, $\lim_{t\to 0} (T_{t}f-f)/t$ exists in $L^{2}(\pi)$ and
\begin{align*}
Af=\lambda f-G=\lambda f-R_{\lambda}f=\mathfrak{L}f.
\end{align*} 
To prove that also $dom(A)\subset D$, we use that for $\chi\in dom (A)$, there trivially exists $f\in L^{2}(\pi)$ with $A\chi=f$. Notice that  by \cref{thm:res-eq}, we can solve the resolvent equation for $\lambda>0$ large enough, 
\begin{align*}
R_{\lambda}\tilde{\chi}=\lambda\chi-f,
\end{align*} with right-hand side $\lambda\chi-f\in L^{2}(\pi)\subset \calC^{0}_{2}$, obtaining a solution $\tilde{\chi}\in D$. By the above, we have that $A_{\mid D}=\mathfrak{L}_{\mid D}$, such that $\mathfrak{L}\tilde{\chi}=A\tilde{\chi}$. This yields by inserting in the equation for $\tilde{\chi}$ and since $f=A\chi$, that $A(\tilde{\chi}-\chi)=\lambda(\tilde{\chi}-\chi)$. As $\lambda>0$, by uniqueness of the solution of the resolvent equation for the generator $A$, we obtain $\tilde{\chi}=\chi$. Thus with the equation for $\tilde{\chi}$ this yields $\chi\in D$ and $\mathfrak{L}\chi=f$. 
\end{proof}
\begin{corollary}\label{cor:good-P-eq}
Let $f\in L^{2}(\pi)$ with $\langle f\rangle_{\pi}=0$. Then there exists a unique solution $\chi\in D$ of the Poisson equation $\mathfrak{L}\chi=f$ such that $\langle\chi\rangle_{\pi}=0$. 
\end{corollary}
\begin{proof}
This follows from the $L^{2}(\pi)$-spectral gap estimates. We can solve the Poisson equation in $L^{2}(\pi)$ for the given right-hand side $f\in L^{2}(\pi)$ with $\langle f\rangle_{\pi}=0$. The solution is explicitly given by $\chi=\int_{0}^{\infty}T_{t}fdt\in L^{2}(\pi)$.\\ We check that $\chi$ is indeed a solution. By \cite[Proposition 1.1.5 part a)]{Ethier1986}, we have that for $f\in L^{2}(\pi)$,  $\int_{0}^{t}T_{s}fds\in dom (A)$ and 
\begin{align*}
T_{t}f-f=A\int_{0}^{t}T_{s}fds,
\end{align*} where $(A,dom(A))$ denotes again the generator of $(T_{t})$ on $L^{2}(\pi)$. By the $L^{2}$-spectral gap estimates and $\langle f\rangle_{\pi}=0$, we obtain that $(\int_{0}^{t}T_{s}fds)_{t}$ converges in $L^{2}(\pi)$ for $t\to\infty$ to a limit $\chi$, and that $(T_{t}f)_{t}$ converges to zero in $L^{2}(\pi)$ for $t\to\infty$. Hence, since $A$ is a closed operator (cf.~\cite[Corollary 1.1.6]{Ethier1986}), we obtain in the limit $t\to\infty$, that $f=A\int_{0}^{\infty} T_{t}fdt=A\chi$ and $\chi\in dom(A)$. Now, using $dom(A)=D$ and $(A,D)=(\mathfrak{L},D)$ by \cref{thm:generator}, this yields $\chi\in D$ and $\mathfrak{L}\chi=f$.
\end{proof}
\noindent Thanks to the regularity of the density of the invariant measure $\pi$, we can finally define the mean of the singular drift $F$ under $\pi$, $\langle F\rangle_{\pi}=\langle F,\rho_{\infty}\rangle_{\lambda}$, respectively the product $F\cdot\rho_{\infty}$.

\begin{lemma}\label{thm:def-F-pi-int}
Let $\rho_{\infty}$ be the density of $\pi$. Let $\langle F\rangle_{\pi}=(\langle F^{i}\rangle_{\pi})_{i=1,...,d}$ for
\begin{align*}
\langle F^{i}\rangle_{\pi}&=(F^{i}\cdot\rho_{\infty})(\mathbf{1})\\&:=[(F^{i}\cdot \rho^{\sharp}_{\infty})+(F^{i}\reso I_{\infty} (\nabla\cdot F))\cdot\rho_{\infty}+C_{1}(\rho_{\infty},I_{\infty}(\nabla\cdot F),F^{i})](\mathbf{1}),
\end{align*} where $\mathbf{1}\in C^{\infty}(\mathbb{T}^{d})$ is the constant test function and $C_{1}$ denotes the paraproduct commutator defined in \eqref{eq:para-comm}.\\ Then, $\langle F^{i}\rangle_{\pi}$ is well-defined and continuous, that is, $\langle F^{m}\rangle_{\pi}\to\langle F\rangle_{\pi}$ for $F^{m}\to F$ in $\mathcal{X}^{\beta,\gamma}_{\infty}(\mathbb{T}^{d})$. Moreover, the following Lipschitz bound holds true
\begin{align*}
\norm{F\cdot\rho_{\infty}}_{\calC^{\beta}(\mathbb{T}^{d})}\lesssim\norm{F}_{\mathcal{X}^{\beta,\gamma}_{\infty}(\mathbb{T}^{d})}(1+\norm{F}_{\mathcal{X}^{\beta,\gamma}_{\infty}(\mathbb{T}^{d})})[\norm{\rho_{\infty}}_{\calC^{\alpha+\beta-1}}+\norm{\rho_{\infty}^{\sharp}}_{\calC^{2(\alpha+\beta-1)}}].
\end{align*}
\end{lemma}
\begin{proof}
The proof follows directly from \cref{thm:ex-fp-d} and \cref{cor:inv-d}. 
\end{proof}
\end{section}

\begin{section}{Solving the Poisson equation with singular right-hand side}\label{sec:s-P-eq}

\noindent To prove the central limit theorem for the solution of the martingale problem $X$, we utilize the classical approach of decomposing the additive functional in terms of a martingale and a boundary term, using the solution of the Poisson equation for $\mathfrak{L}$ with singular right-hand side $F-\langle F\rangle_{\pi}$. For solving the Poisson equation in \cref{thm:poisson-eq} below, \cref{cor:good-P-eq} is not applicable, as $F$ is a distribution and therefore not an element of $L^{2}(\pi)$. Consider an approximation $(F^{m})\subset C^{\infty}(\mathbb{T}^{d})$ with $F^{m}\to F$ in $\mathcal{X}^{\beta,\gamma}_{\infty}(\mathbb{T}^{d})$. Then, we can apply \cref{cor:good-P-eq} for the right-hand sides $F^{m}-\langle F^{m}\rangle_{\pi}\in L^{2}(\pi)$, $m\in\N$. This way we obtain solutions $\chi^{m}=(\chi^{m,i})_{i=1,...,d}\in D^{d}\subset L^{2}(\pi)^{d}$ of the Poisson equations 
\begin{align}\label{eq:Poisson-m}
(-\mathfrak{L})\chi^{m,i}=F^{m,i}-\langle F^{m,i}\rangle_{\pi}
\end{align} for $m\in\N$.\\ 
In this section, we show that the sequence $(\chi^{m})_{m}$ converges in a space of sufficient regularity to a the limit $\chi$ that indeed solves the Poisson equation 
\begin{align}\label{eq:Poisson-limit}
(-\mathfrak{L})\chi=F-\langle F\rangle_{\pi}.
\end{align}
Let us define the space $\mathcal{H}^{1}(\pi)$ as in \cite[Section 2.2]{klo},
\begin{align}
\mathcal{H}^{1}(\pi):=\{f\in D\mid  \norm{f}_{\mathcal{H}^{1}(\pi)}^{2}:=\langle(-\mathfrak{L})f,f\rangle_{\pi}<\infty\},
\end{align} which is the Sobolev space for the operator $\mathfrak{L}$ with respect to $L^{2}(\pi)$. Its dual is defined by
\begin{align}
\mathcal{H}^{-1}(\pi):=\{F:\mathcal{H}^{1}(\pi)\to\R\mid F\text{ linear with } \norm{F}_{\mathcal{H}^{-1}(\pi)}:=\sup_{\norm{f}_{\mathcal{H}^{1}(\pi)}=1}\abs{F(f)}<\infty\}.
\end{align}
The space $\mathcal{H}^{1}(\pi)$ is related to the quadratic variation of Dynkin's martingale, see \cite[Section 2.4]{klo}, which motivates the definition.\\ 
To prove convergence of $(\chi^{m})_{m}$ in $L^{2}(\pi)^{d}$, we first establish in \cref{cor:H1-conv} convergence of $(\chi^{m})_{m}$ in the space $\mathcal{H}^{1}(\pi)^{d}$ and utilize a Poincaré-type bound on the operator $\mathfrak{L}$. A standard argument as in \cite[Property 2.4]{Guionnet02} shows that the $L^{2}(\pi)$-spectral gap estimates from \cref{cor:L2-spectral-gap} for the constant $K=1$, imply the Poincaré estimate for the operator $\mathfrak{L}$:
\begin{align*}
\norm{f-\langle f\rangle_{\pi}}_{L^{2}(\pi)}^{2}\leqslant \mu \langle (-\mathfrak{L})f,f\rangle_{\pi}=\mu \norm{f}_{\mathcal{H}^{1}(\pi)}^{2}, \quad\text{ for all } f\in D.
\end{align*} 
In general, the constant $K>0$ in the spectral gap estimates from \cref{cor:L2-spectral-gap} does not need to satisfy $K=1$ and the above argument breaks down for $K\neq 1$.
Hence, we show below in \eqref{eq:poincare} that $\norm{f-\langle f\rangle_{\pi}}_{L^{2}(\pi)}^{2}\leqslant C \norm{f}_{\mathcal{H}^{1}(\pi)}^{2}$ holds true for some constant $C>0$. That constant may differ from the constant $\mu$ and may not be optimal, but the bound suffices for our purpose of concluding on $L^{2}(\pi)^{d}$ convergence given $\mathcal{H}^{1}(\pi)^{d}$ convergence of $(\chi^{m})_{m}$.\\ An optimal estimate, that however applies for a much more general situation of weak Poincaré inequalities and slower than exponential convergences, can be found in \cite[Theorem 2.3]{RoecknerWang}.\\
The $\mathcal{H}^{1}(\pi)^{d}$ convergence of $(\chi^{m})_{m}$ follows from $\mathcal{H}^{-1}(\pi)^{d}$-convergence of $(F^{m})_{m}$ for the approximating sequence $F^{m}\to F$ in $\mathcal{X}^{\beta,\gamma}_{\infty}(\mathbb{T}^{d})$. Convergence of $(F^{m})_{m}$ in $\mathcal{H}^{-1}(\pi)^{d}$ is established in \cref{thm:dual}. The following lemma is an auxiliary result, which proves that the semi-norms in $\mathcal{H}^{1}(\pi)$ and the homogeneous Besov space $\dot{B}^{\alpha/2}_{2,2}(\mathbb{T}^{d})$, cf.~\eqref{eq:p-bs2},  are equivalent. 

\begin{lemma}\label{lem:cdc-fL}
Let $\alpha\in (1,2]$. 
Define the carré-du-champ operator of the generalized fractional Laplacian as $\Gamma^{\alpha}_{\nu}(f)=\Gamma^{\alpha}_{\nu}(f,f):=\frac{1}{2}((-\La) f^{2}-2f(-\La) f)$. 
Then, there exist constants $c,C>0$, such that for all $f\in \dot{B}^{\alpha/2}_{2,2}(\mathbb{T}^{d})$,
\begin{align}\label{eq:eq-norms}
c\norm{f}_{\dot{B}^{\alpha/2}_{2,2}(\mathbb{T}^{d})}^{2}\leqslant\langle \Gamma^{\alpha}_{\nu}(f)\rangle_{\lambda}\leqslant C \norm{f}_{\dot{B}^{\alpha/2}_{2,2}(\mathbb{T}^{d})}^{2}.
\end{align} 
\end{lemma}
\begin{proof}
By \cite[part (v) of Theorem, Section 3.5.4]{Schmeisser1987} we obtain that the periodic Lizorkin space $F_{2,2}^{s}(\mathbb{T}^{d})$ coincides with the periodic Bessel-potential space $H^{s}(\mathbb{T}^{d})$. Furthermore $F_{2,2}^{s}(\mathbb{T}^{d})$ coincides with $B_{2,2}^{s}(\mathbb{T}^{d})$ (cf. \cite[Section 3.5.1, Remark 4]{Schmeisser1987}). Thus, we obtain that in particular 
\begin{align*}
\dot{B}^{s}_{2,2}(\mathbb{T}^{d})=\dot{H}^{s}(\mathbb{T}^{d}).
\end{align*} 
It remains to show \eqref{eq:eq-norms} with $\dot{B}^{\alpha/2}_{2,2}(\mathbb{T}^{d})$ replaced by $\dot{H}^{s}(\mathbb{T}^{d})$. We prove the claim for $\alpha\in (1,2)$, for $\alpha=2$ the proof is similar. To that aim, we calculate, using the definition of $\La$ for a Schwartz function $f\in \mathcal{S}(\mathbb{T}^{d})$ and $\psi^{\alpha}_{\nu}(0)=0$,
\begin{align*}
\langle\Gamma^{\alpha}_{\nu}(f)\rangle_{\lambda}=\int_{\mathbb{T}^{d}} \Gamma^{\alpha}_{\nu}(f) (x)dx&=\F_{\mathbb{T}^{d}}(\Gamma^{\alpha}_{\nu}(f))(0)\\&=\frac{1}{2}\F_{\mathbb{T}^{d}}((-\La) f^{2})(0)-\F_{\mathbb{T}^{d}}(f(-\La) f)(0)
\\&=-\frac{1}{2}\psi^{\alpha}_{\nu}(0)(\hat{f}\ast\hat{f})(0)+(\hat{f}\ast\psi^{\alpha}_{\nu}\hat{f})(0)
\\&=\sum_{k\in\mathbb{Z}^{d}}\hat{f}(-k)\hat{f}(k)\psi^{\alpha}_{\nu}(k)\\&=\sum_{k\in\mathbb{Z}^{d}}\abs{\hat{f}(k)}^{2}\psi^{\alpha}_{\nu}(k).
\end{align*}
By \cref{ass} on the spherical component of the jump measure $\nu$, we obtain, that there exist constants $c,C>0$ with
\begin{align*}
c\abs{k}^{\alpha}\leqslant\psi^{\alpha}_{\nu}(k)=\int_{S}\abs{\langle k,\xi\rangle}^{\alpha}\nu(d\xi)\leqslant C\abs{k}^{\alpha}.
\end{align*} Thus it follows that
\begin{align*}
c\norm{f}_{\dot{H}^{\alpha/2}(\mathbb{T}^{d})}^{2}=c\sum_{k\in\mathbb{Z}^{d}}\abs{k}^{\alpha}\abs{\hat{f}(k)}^{2}\leqslant\langle\Gamma^{\alpha}_{\nu}(f)\rangle_{\lambda}\leqslant C\sum_{k\in\mathbb{Z}^{d}}\abs{k}^{\alpha}\abs{\hat{f}(k)}^{2}=C\norm{f}_{\dot{H}^{\alpha/2}(\mathbb{T}^{d})}^{2}.
\end{align*}
By a density argument, the claim follows for all $f\in \dot{H}^{\alpha/2}(\mathbb{T}^{d})=\dot{B}^{\alpha/2}_{2,2}(\mathbb{T}^{d})$.

\end{proof}
\begin{theorem}\label{thm:dual}
Let $F\in\mathcal{X}^{\beta,\gamma}_{\infty}(\mathbb{T}^{d})$ for $\beta\in (\frac{2-2\alpha}{3},0)$ and $\alpha\in (1,2]$.\\
Then, equivalence of the semi-norms $\norm{\cdot}_{\mathcal{H}^{1}(\pi)}\simeq \norm{\cdot}_{\dot{B}^{\alpha/2}_{2,2}(\mathbb{T}^{d})}$ follows and $\overline{F}:=F-\langle F\rangle_{\pi}\in\mathcal{H}^{-1}(\pi)^{d}$. In particular, $F^{m}\to F$ in $\mathcal{X}^{\beta,\gamma}_{\infty}(\mathbb{T}^{d})$ implies $\overline{F}^{m}\to\overline{F}$ in $\mathcal{H}^{-1}(\pi)^{d}$. 
\end{theorem}
\begin{proof}
By invariance of $\pi$ we obtain $\langle \mathfrak{L}g\rangle_{\pi}=0$ for $g\in D$, because for $g\in D$, $(\frac{d}{dt}T_{t})_{\mid t=0}f=\mathfrak{L}f\in L^{2}(\pi)$. We now apply this for $g=f^{2}$ for which we need to check that if $f\in D$, then $\mathfrak{L}f^{2}$ is well-defined and $\mathfrak{L}f^{2}\in L^{1}(\pi)$. This follows by calculating
\begin{align*}
f^{2}=(f^{\sharp}+\nabla f\para I_{\lambda}(F))^{2}=g^{\sharp}+g^{\prime}\para I_{\lambda}(F),
\end{align*} where 
\begin{align*}
g^{\sharp}&=(f^{\sharp})^{2}+2f^{\sharp}\reso(\nabla f\para I_{\lambda}(F))+2f^{\sharp}\arap(\nabla f\para I_{\lambda}(F))\\&\qquad\qquad\qquad+(\nabla f\para I_{\lambda}(F))\reso (\nabla f\para I_{\lambda}(F))\in\calC^{2\theta-1}_{1}(\mathbb{T}^{d})
\end{align*} and 
\begin{align*}
g^{\prime}=2f^{\sharp}\para \nabla f+\nabla f\para I_{\lambda}(F)\para \nabla f+I_{\lambda}(F)\para\nabla f\para I_{\lambda}(F)\in(\calC^{\theta-1}_{1}(\mathbb{T}^{d}))^{d}.
\end{align*}
Hence, we conclude that for $f\in D$, $f^{2}$ admits a paracontrolled structure with $g^{\sharp}\in\calC^{2\theta-1}_{1}(\mathbb{T}^{d})$ and $g^{\prime}\in(\calC^{\theta-1}_{1}(\mathbb{T}^{d}))^{d}$, such that $\mathfrak{L}f^{2}$ is well-defined and 
\begin{align*}
\mathfrak{L}f^{2}=2f \mathfrak{L}f+2\Gamma^{\alpha}_{\nu}(f)=2\lambda f^{2}-2fR_{\lambda}f+2\Gamma^{\alpha}_{\nu}(f)\in L^{1}(\pi).
\end{align*} 
Herein we used that $2\lambda f^{2}-2fR_{\lambda}f\in L^{1}(\pi)$ as $f\in D$ and $\Gamma^{\alpha}_{\nu}(f)=\Gamma^{\alpha}_{\nu}(f,f)=\frac{1}{2}(\La f^{2}-2f\La f)\in L^{1}(\pi)$ for $f\in\calC^{\theta}_{2}(\mathbb{T}^{d})$ by \cref{lem:cdc-fL} as $\theta$ can be chosen close to $\alpha+\beta$, such that $\theta>\alpha/2$.\\ Analogously, if we denote the domain of $\mathfrak{L}$ with integrability $p$ by $D_{p}$, then for $f,g\in D_{2}$, we concluded that $f\cdot g\in D_{1}$, which in particular implies that the carré-du-champ operator 
\begin{align*}
\Gamma^{\mathfrak{L}}(f,g)=\frac{1}{2}(\mathfrak{L}(fg)-f\mathfrak{L}g-g\mathfrak{L}f)\in L^{1}(\pi)
\end{align*} for $f,g\in D$ is well-defined in $L^{1}(\pi)$.\\
Applying invariance of $\pi$ for $g=f^{2}$, we can add $\frac{1}{2}\langle\mathfrak{L}f^{2}\rangle_{\pi}=0$ yielding  
\begin{align*}
\norm{f}_{\mathcal{H}^{1}(\pi)}^{2}=\langle (-\mathfrak{L})f,f\rangle_{\pi}=\langle\Gamma^{\mathfrak{L}}(f)\rangle_{\pi}=\langle \Gamma^{\alpha}_{\nu}(f)\rangle_{\pi}.
\end{align*} where $\Gamma^{\mathfrak{L}}(f)=\frac{1}{2}\mathfrak{L}f^{2}-f\mathfrak{L}f=\Gamma^{\alpha}_{\nu}(f)$.
Thus, we obtain
\begin{align}\label{eq:H1-spaces}
\norm{f}_{\mathcal{H}^{1}(\pi)}^{2}=\langle\Gamma^{\alpha}_{\nu}(f)\rangle_{\pi}\simeq\langle\Gamma^{\alpha}_{\nu}(f)\rangle_{\lambda}\simeq\norm{f}_{\dot{B}^{\alpha/2}_{2,2}}^{2},
\end{align} where $\simeq$ denotes that the norms are equivalent.\\ Here, we used that absolute continuity of $\pi$ with respect to the Lebesgue-measure, with density $\rho_{\infty}$ that is uniformly bounded from above and from below, away from zero by \cref{cor:inv-d}. Moreover, note that the carré-du-champ is non-negative, $\Gamma^{\alpha}_{\nu}(f)\geqslant 0$. Furthermore we utilized \eqref{eq:eq-norms} from \cref{lem:cdc-fL}.\\ 
Thus applying the duality estimate from \cref{lem:duality} (for functions $f-\langle f\rangle_{\lambda}, g-\langle g\rangle_{\lambda}$ to obtain the result for the homogeneous Besov spaces),
we get for $\overline{F}:=F-\langle F\rangle_{\pi}$ with mean $\langle F\rangle_{\pi}$ from \cref{thm:def-F-pi-int},
\begin{align*}
\abs{\langle \overline{F}^{i}, g\rangle_{\pi}}&=\abs{\langle\overline{F}^{i}\rho_{\infty}, g\rangle }
\\&\lesssim \norm{\overline{F}^{i}\rho_{\infty}}_{\dot{B}^{-\alpha/2}_{2,2}(\mathbb{T}^{d})}\norm{g}_{\dot{B}^{\alpha/2}_{2,2}(\mathbb{T}^{d})}\\&\lesssim \norm{\overline{F}^{i}\rho_{\infty}}_{\dot{B}^{\beta}_{2,2}(\mathbb{T}^{d})}\norm{g}_{\dot{B}^{\alpha/2}_{2,2}(\mathbb{T}^{d})} \\&\lesssim\norm{\overline{F}^{i}\rho_{\infty}}_{\dot{B}^{\beta}_{2,2}(\mathbb{T}^{d})}\norm{g}_{\mathcal{H}^{1}(\pi)},
\end{align*} for $i=1,...,d$, using $\beta>-\alpha/2$ and \eqref{eq:H1-spaces}. Hence, we find 
\begin{align*}
\norm{\overline{F}^{i}}_{\mathcal{H}^{-1}(\pi)}&\lesssim \norm{\overline{F}^{i}\rho_{\infty}}_{\dot{B}^{\beta}_{2,2}(\mathbb{T}^{d})}\\&\lesssim\norm{\overline{F}^{i}\rho_{\infty}}_{B^{\beta}_{2,2}(\mathbb{T}^{d})}\\&\lesssim\norm{\overline{F}^{i}\rho_{\infty}}_{\calC^{\beta+(1-\gamma)\alpha}(\mathbb{T}^{d})}\\&\lesssim\norm{F}_{\mathcal{X}^{\beta,\gamma}_{\infty}(\mathbb{T}^{d})}(1+\norm{F}_{\mathcal{X}^{\beta,\gamma}_{\infty}(\mathbb{T}^{d})})[\norm{\rho_{\infty}}_{\calC^{\alpha+\beta-1}}+\norm{\rho_{\infty}^{\sharp}}_{\calC^{2(\alpha+2\beta-1)}}],
\end{align*}
where the estimate for the product of $\overline{F}^{i}$ and $\rho_{\infty}$ follows from \cref{thm:def-F-pi-int}.\\
This proves that $\overline{F}\in\mathcal{H}^{-1}(\pi)^{d}$. Convergence follows by the same estimate.
\end{proof}
\begin{corollary}\label{cor:H1-conv}
Let $F\in\mathcal{X}^{\beta,\gamma}_{\infty}$ and $F^{m}\in C^{\infty}(\mathbb{T}^{d})$ with $F^{m}\to F$ in $\mathcal{X}^{\beta,\gamma}_{\infty}(\mathbb{T}^{d})$. Let $\chi^{m}=(\chi^{m,i})_{i=1,\dots,d}\in L^{2}(\pi)^{d}$ denote the unique solution of 
\begin{align*}
(-\mathfrak{L})\chi^{m,i}=F^{m,i}-\langle F^{m,i}\rangle_{\pi}=:\overline{F}^{m,i}
\end{align*} with $\langle\chi^{m,i}\rangle_{\pi}=0$. Then $(\chi^{m})_{m}$ converges in $\mathcal{H}^{1}(\pi)^{d}\cap L^{2}(\pi)^{d}$ to a limit $\chi$. 
\end{corollary}
\begin{proof} 
Convergence in $\mathcal{H}^{1}(\pi)$ follows from the estimate
\begin{align*}
\norm{\chi^{m,i}-\chi^{m',i}}_{\mathcal{H}^{1}(\pi)}^{2}&=\langle(-\mathfrak{L})(\chi^{m,i}-\chi^{m',i}),\chi^{m,i}-\chi^{m',i}\rangle_{\pi}\\&=\langle \overline{F}^{m,i}-\overline{F}^{m',i},\chi^{m,i}-\chi^{m',i}\rangle_{\pi}
\\&\leqslant\norm{\overline{F}^{m,i}-\overline{F}^{m',i}}_{\mathcal{H}^{-1}(\pi)}\norm{\chi^{m,i}-\chi^{m',i}}_{\mathcal{H}^{1}(\pi)}.
\end{align*}
Thus we obtain
\begin{align*}
\norm{\chi^{m,i}-\chi^{m',i}}_{\mathcal{H}^{1}(\pi)}&\leqslant \norm{\overline{F}^{m,i}-\overline{F}^{m',i}}_{\mathcal{H}^{-1}(\pi)}.
\end{align*}
And indeed the $\mathcal{H}^{-1}(\pi)$-norm on the right-hand side is small, when $m,m'$ are close, by \cref{thm:dual}.\\ It remains to conclude on $L^{2}(\pi)$ convergence.  
By \cref{thm:dual}, we also obtain the seminorm equivalences, $\norm{\cdot}_{\mathcal{H}^{1}(\pi)}\simeq\norm{\cdot}_{\dot{H}^{\alpha/2}(\mathbb{T}^{d})}\simeq\norm{\cdot}_{\dot{B}^{\alpha/2}_{2,2}(\mathbb{T}^{d})}$. 
Combining with the fractional Poincaré inequality on the torus, 
\begin{align*}
\norm{u-\langle u\rangle_{\lambda}}_{L^{2}}^{2}=\sum_{k\in\mathbb{Z}^{d}\setminus\{0\}}\abs{\hat{u}(k)}^{2}\leqslant\sum_{k\in\mathbb{Z}^{d}\setminus\{0\}}\abs{k}^{\alpha}\abs{\hat{u}(k)}^{2}=\norm{u}_{\dot{H}^{\alpha/2}(\mathbb{T}^{d})}^{2},
\end{align*} with Lebesgue measure $\lambda$ on $\mathbb{T}^{d}$, we can thus estimate
\begin{align}\label{eq:poincare}
\norm{\chi-\langle\chi\rangle_{\lambda}}_{L^{2}(\pi)}\lesssim\norm{\chi-\langle\chi\rangle_{\lambda}}_{L^{2}(\lambda)}\leqslant\norm{\chi}_{\dot{H}^{\alpha/2}(\mathbb{T}^{d})}\lesssim\norm{\chi}_{\mathcal{H}^{1}(\pi)}.
\end{align} 
Furthermore, as $\langle\chi\rangle_{\pi}=0$, we obtain $\norm{\chi-\langle\chi\rangle_{\lambda}}_{L^{2}(\pi)}^{2}=\norm{\chi}_{L^{2}(\pi)}^{2}+\langle\chi\rangle_{\lambda}^{2}$. Together, we thus find
\begin{align}
\norm{\chi}_{L^{2}(\pi)}^{2}\lesssim\norm{\chi}_{L^{2}(\pi)}^{2}+\langle\chi\rangle_{\lambda}^{2}\lesssim\norm{\chi}_{\mathcal{H}^{1}(\pi)}^{2}.
\end{align} 
In particular, we conclude that $\mathcal{H}^{1}(\pi)$-convergence implies $L^{2}(\pi)$-convergence of the sequence $(\chi^{m})$.
\end{proof}

\begin{theorem}\label{thm:poisson-eq}
Let $(F^{m})_{m}$, $(\chi^{m})_{m}$ and $\chi$ be as in \cref{cor:H1-conv}.\\ Then, $(\chi^{m})_{m}$ converges to $\chi$ in $(\calC^{\theta}_{2}(\mathbb{T}^{d}))^{d}$, $\theta\in ((1-\beta)/2,\alpha+\beta)$ and there exists $\lambda>0$, such that 
\begin{align}
\chi=\chi^{\sharp}+\nabla \chi\para I_{\lambda}(\overline{F})
\end{align} for $\chi^{\sharp}\in(\calC^{2\theta-1}_{2}(\mathbb{T}^{d}))^{d}$.\\ Furthermore, the limit $\chi$ solves the singular Poisson equation with singular right-hand side $\overline{F}$,  
\begin{align}
(-\mathfrak{L})\chi=\overline{F}.
\end{align}
\end{theorem}
\begin{proof}
Trivially, for $\lambda>0$, $\chi^{m}$ solves the resolvent equation
\begin{align*}
R_{\lambda}\chi^{m}=(\lambda-\mathfrak{L})\chi^{m}=\lambda\chi^{m}+\overline{F}^{m}
\end{align*} with right-hand side $G^{m}:=\lambda\chi^{m}+\overline{F}^{m}$. The right-hand sides $(G^{m})$ converge in $(\calC^{\beta}_{2}(\mathbb{T}^{d}))^{d}$ to $G=\lambda\chi+\overline{F}$, because $\chi^{m}\to\chi$ in $L^{2}(\pi)^{d}$ by \cref{cor:H1-conv} and, thanks to the equivalence of $\pi$ and the Lebesgue measure $\lambda_{\mathbb{T}^{d}}$, thus also in $L^{2}(\lambda_{\mathbb{T}^{d}})^{d}$. Choosing $\lambda>1$ big enough, by \cref{thm:res-eq}, we can solve the resolvent equation 
\begin{align}\label{eq:g}
R_{\lambda}g^{i}=G^{i}=G^{\sharp,i}+G^{\prime,i}\para F,
\end{align} with $G^{\sharp,i}:=\lambda\chi^{i}\in L^{2}(\lambda)\subset\calC^{0}_{2}(\mathbb{T}^{d})$ and $G^{\prime,i}:=(1-\langle F^{i}\rangle_{\pi})e_{i}\in\calC^{\alpha+\beta-1}(\mathbb{T}^{d})$. Thereby we  obtain a paracontrolled solution $g^{i}\in\mathcal{D}^{\theta}_{2}$ for $\theta<\alpha+\beta$, with $g^{i}=g^{\sharp,i}+\nabla g^{i}\para I_{\lambda}(F)$, $g^{\sharp,i}\in\calC^{2\theta-1}_{2}(\mathbb{T}^{d})$ and $I_{\lambda}(F):=\int_{0}^{\infty}e^{-\lambda t}P_{t}F dt \in \calC^{\alpha+\beta}(\mathbb{T}^{d})$. By continuity of the solution map for the resolvent equation, we obtain convergence of $\chi^{m,i}\to g^{i}$ in $\mathcal{D}^{\theta}_{2}$ for $m\to\infty$. Convergence of $(\chi^{m})$ to $g$ in $(\mathcal{D}^{\theta}_{2})^{d}$ in particular implies convergence in $L^{2}(\lambda_{\mathbb{T}^{d}})^{d}$ and thus in $L^{2}(\pi)^{d}$, which implies that almost surely $g=\chi$ and hence, by \eqref{eq:g}, that $\chi\in (\mathcal{D}^{\theta}_{2})^{d}$ solves $(-\mathfrak{L})\chi=\overline{F}$.
\end{proof}

\end{section}

\begin{section}{Fluctuations in the Brownian and pure Lévy noise case}\label{sec:CLT}
In this section, we prove the central limit \cref{thm:main-thm-ph} for the diffusion $X$ with periodic coefficients. In the following, we again explicitly distinguish between $X$ and the projected process $X^{\mathbb{T}^{d}}$. Of course, the central limit theorem in particular implies that for $t>0$, $\frac{1}{n}X_{nt}\to t \langle F\rangle_{\pi}$ with convergence in probability for $n\to\infty$, i.e.~a weak law of large numbers. The central limit theorem then quantifies the fluctuations around the mean $t \langle F\rangle_{\pi}$.\\
Due to ergodicity of $\pi$, it follows by the von Neumann ergodic theorem that, if the projected process is started in $X_{0}^{\mathbb{T}^{d}}\sim\pi$,   
$\frac{1}{n}\int_{0}^{nt}b(X_{s}^{\mathbb{T}^{d}})ds\to t\langle b\rangle_{\pi}$ in $L^{2}(\p_{\pi})$ as $n\to\infty$  for $b\in L^{\infty}(\mathbb{T}^{d})$. As $\p_{\pi}=\int_{\mathbb{T}^{d}}\p_{x}\pi(dx)$, this implies in particular the convergence (along a subsequence) in $L^{2}(\p_{x})$ for $\pi$-almost all $x$.\\ The pointwise spectral gap estimates yield the following slightly stronger ergodic theorem for the process started in $X_{0}^{\mathbb{T}^{d}}=x$ for any $x\in\mathbb{T}^{d}$. In particular, in the periodic homogenization result for the PDE, \cref{cor:PDEhomog} below, pointwise convergence (for every $x\in\mathbb{T}^{d}$) of the PDE solutions can be proven.

\begin{lemma}\label{lem:ergodic-thm}
Let $b\in L^{\infty}(\mathbb{T}^{d})$ and $x\in\mathbb{T}^{d}$. Let $X^{\mathbb{T}^{d}}$ be the projected solution of the $\mathcal{G}=\partial_{t}+\mathfrak{L}$- martingale problem on the torus $\mathbb{T}^{d}$ started in $X_{0}^{\mathbb{T}^{d}}=x\in\mathbb{T}^{d}$.\\ Then the following convergence holds in $L^{2}(\p)$: 
\begin{align*}
\frac{1}{n}\int_{0}^{nt}b(X_{s}^{\mathbb{T}^{d}})ds\to t\langle b\rangle_{\pi}.
\end{align*}
\end{lemma}
\begin{proof}
Without loss of generality, we assume that $\langle b\rangle_{\pi}=0$, otherwise we subtract the mean. With the Markov property we obtain
\begin{align*}
\norm[\bigg]{\frac{1}{n}\int_{0}^{nt}b(X_{s}^{\mathbb{T}^{d}})ds}_{L^{2}(\p)}^{2}&=\frac{1}{n^2}\int_{0}^{nt}\int_{0}^{nt}\E[b(X_{s}^{\mathbb{T}^{d}})b(X_{r}^{\mathbb{T}^{d}})]dsdr\\&=\frac{2}{n^2}\int_{0}^{nt}\int_{0}^{nt}\mathbf{1}_{s\leqslant r}\E\Bigl[b(X_{s}^{\mathbb{T}^{d}})\E_{s}[b(X_{r}^{\mathbb{T}^{d}})]\Bigr]dsdr
\end{align*}
Using the spectral gap estimate \eqref{eq:p-sp-est}, we can estimate
\begin{align*}
\abs[\bigg]{\frac{2}{n^2}\int_{0}^{nt}\int_{0}^{nt}\mathbf{1}_{s\leqslant r}T_{s}(bT_{r-s}b)(x)dsdr}&\leqslant\frac{2K^{2}\norm{b}_{L^{\infty}}^{2}}{n^2}\int_{0}^{nt}\int_{0}^{nt}e^{-\mu s}e^{-\mu(r-s)}dsdr
\\&=\frac{tK^{2}\norm{b}_{L^{\infty}}^{2}}{n\mu}(1-e^{-\mu nt})\to 0,
\end{align*} for $n\to\infty$. 
\end{proof}
\begin{theorem}\label{thm:main-thm-ph}
Let $\alpha\in (1,2]$ and $F\in\calC^{\beta}(\mathbb{T}^{d})$ for $\beta\in (\frac{1-\alpha}{2},0)$ or $F\in\mathcal{X}^{\beta,\gamma}_{\infty}(\mathbb{T}^{d})$ for $\beta\in (\frac{2-2\alpha}{3},\frac{1-\alpha}{2}]$ and $\gamma\in (\frac{2\beta+2\alpha-1}{\alpha},1)$. 
Let $X$ be the solution of the $\mathcal{G}=(\partial_{t}+\mathfrak{L})$-martingale problem started in $X_{0}=x\in\R^{d}$.\\
In the case $\alpha=2$ and $L=B$ for a standard Brownian motion $B$, the following functional central limit theorem holds:
\begin{align*}
\paren[\bigg]{\frac{1}{\sqrt{n}}(X_{nt}-nt\langle F\rangle_{\pi})}_{t\in [0,T]}\Rightarrow \sqrt{D}(W_{t})_{t\in [0,T]},
\end{align*} with convergence in distribution in $C([0,T],\R^{d})$, a $d$-dimensional standard Brownian motion $W$
and constant diffusion matrix $D$ given by
\begin{align*}
D(i,j):=\int_{\mathbb{T}^{d}} (e_{i}+\nabla \chi^{i}(x))^{T}(e_{j}+\nabla \chi^{j}(x))\pi(dx)
\end{align*} for $i,j=1,\dots,d$ and the $i$-th euclidean unit vector $e_{i}$. Here, $\chi$ solves the singular Poisson equation $(-\mathfrak{L})\chi^{i}=F^{i}-\langle F^{i}\rangle_{\pi}$, $i=1,...,d$, according to \cref{thm:poisson-eq}.\\
In the case $\alpha\in (1,2)$, the following non-Gaussian central limit theorem holds:
\begin{align*}
\paren[\bigg]{\frac{1}{n^{1/\alpha}}(X_{nt}-nt\langle F\rangle_{\pi})}_{t\in [0,T]}\Rightarrow (\tilde{L}_{t})_{t\in [0,T]},
\end{align*} with convergence in distribution in $D([0,T],\R^{d})$, where $\tilde{L}$ is a $d$-dimensional symmetric $\alpha$-stable nondegenerate Lévy process (with generator $-\La$).
\end{theorem}

\begin{proof}
As a byproduct of \cite[Theorem 5.10]{kp-wsmp}, we obtain that there exists a probability space $(\Omega,\F,\p)$ with an $\alpha$-stable symmetric non-degenerate process $L$, such that $X=x+Z+L$, where $Z$ is given by 
\begin{align}\label{eq:Z}
Z_{t}=\lim_{m\to\infty}\int_{0}^{t}F^{m}(X_{s})ds
\end{align} for a sequence $(F^{m})$ of smooth functions $F^{m}$ with $F^{m}\to F$ in $\mathcal{X}^{\beta,\gamma}_{\infty}(\mathbb{T}^{d})$ 
and where the limit is taken in $L^{2}(\p)$, uniformly in $t\in [0,T]$.\\
We write the additive functional $\int_{0}^{\cdot}(\overline{F^{m}})^{\R^{d}}(X_{s})ds=\int_{0}^{\cdot} \overline{F^{m}}(X_{s}^{\mathbb{T}^{d}})ds$ in terms of the periodic solution $\chi^{m}$ of the Poisson equation \eqref{eq:Poisson-m} with right hand side $F^{m}-\langle F^{m}\rangle _{\pi}=:\overline{F^{m}}$, such that
\begin{align}
X_{t}-t\langle F\rangle_{\pi}&=X_{0}+(Z_{t}-t\langle F\rangle_{\pi})+L_{t}\\&=X_{0}+\lim_{m\to\infty}\int_{0}^{t}\overline{F^{m}}(X_{s}^{\mathbb{T}^{d}})ds+L_{t}\label{eq:decomp-ph}\\&=X_{0}+\lim_{m\to\infty}\paren[\big]{[\chi^{m}(X_{0}^{\mathbb{T}^{d}})-\chi^{m}(X_{t}^{\mathbb{T}^{d}})]+M_{t}^{m}}+L_{t}
\label{eq:mart-ph}\\&=X_{0}+[\chi(X_{0}^{\mathbb{T}^{d}})-\chi(X_{t}^{\mathbb{T}^{d}})]+M_{t}+L_{t}.\label{eq:decomp-ph2}
\end{align} 
Here, the limit is again taken in $L^{2}(\p)$ and $\chi$ is the solution of the Poisson equation \eqref{eq:Poisson-limit} with right-hand side $\overline{F}$, which exists by \cref{thm:poisson-eq}.\\ To justify \eqref{eq:decomp-ph}, we use the convergence from \eqref{eq:Z} and $\langle F\rangle_{\pi}=\lim_{m\to\infty}\langle F^{m}\rangle_{\pi}$ by \cref{thm:def-F-pi-int}. In \eqref{eq:mart-ph}, we applied Itô's formula to $(\chi^{m})^{\R^{d}}(X_{t})$ for $m\in\N$. For the equality \eqref{eq:decomp-ph2}, we utilized that $\chi^{m}\to\chi$ in $L^{\infty}(\mathbb{T}^{d})$ by \cref{thm:poisson-eq} 
and that the sequence of martingales $(M^{m})$ converges in $L^{2}(\p)$ uniformly in time in $[0,T]$ to the martingale $M$. Here, for $\alpha\in (1,2)$, the martingales are given by (notation: $[y]:=y\mod\mathbb{Z}^{d}=\iota(y)$)
\begin{align*}
M_{t}^{m}=\int_{0}^{t}\int_{\R^{d}\setminus\{0\}}[\chi^{m}(X_{s-}^{\mathbb{T}^{d}}+[y]))-\chi^{m}(X_{s-}^{\mathbb{T}^{d}})]\hat{\pi}(ds,dy),
\end{align*} where $\hat{\pi}(ds,dy)=\pi(ds,dy)-ds\mu(dy)$ is the compensated Poisson random measure associated to $L$. $M$ is given by an analogue expression, where we replace $\chi^{m}$ by $\chi$.\\ In the Brownian noise case, $\alpha=2$, we have that $M_{t}^{m}=\int_{0}^{t}\nabla\chi^{m}(X_{s}^{\mathbb{T}^{d}})\cdot dB_{s}$ and $M_{t}$ is defined analogously with $\chi^{m}$ replaced by $\chi$. Indeed, convergence of the martingales in $L^{2}(\p)$ follows from the convergence of $(\chi^{m})$ to $\chi$ in $\calC^{\theta}_{2}(\mathbb{T}^{d})$ with $\theta\in (1,\alpha+\beta)$ by \cref{thm:poisson-eq}, which in particular implies uniform convergence of $(\chi^{m})$ and $(\nabla\chi^{m})$.\\

\noindent Let now first $\alpha=2$ and $L=B$ for a standard Brownian motion $B$. Then we have by the above, almost surely,
\begin{align*}
\frac{1}{\sqrt{n}}(X_{nt}-nt\langle F\rangle_{\pi})=\frac{1}{\sqrt{n}}X_{0}+\frac{1}{\sqrt{n}}[\chi(X_{0}^{\mathbb{T}^{d}})-\chi(X_{nt}^{\mathbb{T}^{d}})]+\frac{1}{\sqrt{n}}(M_{nt}+B_{nt})
\end{align*} with $M_{t}=\int_{0}^{t}\nabla\chi(X_{s}^{\mathbb{T}^{d}})\cdot dB_{s}$.\\
To obtain the central limit theorem, we will apply the functional martingale central limit theorem, \cite[Theorem 7.1.4]{Ethier1986}, to 
\begin{align*}
\paren[\bigg]{\frac{1}{\sqrt{n}}(M_{nt}+B_{nt})}_{t\in [0,T]}.
\end{align*} 
To that aim, we check the convergence of the quadratic variation 
\begin{align*}
\frac{1}{n}\langle M^{i}+B^{i},M^{j}+B^{j}\rangle_{nt}=\frac{1}{n}\int_{0}^{nt}(\operatorname{Id}+\nabla\chi(X_{s}^{\mathbb{T}^{d}}))^{T}(\operatorname{Id}+\nabla\chi(X_{s}^{\mathbb{T}^{d}})) (i,j)ds
\end{align*} in probability to 
\begin{align*}t\int_{\mathbb{T}^{d}}(\operatorname{Id}+\nabla\chi(x))^{T}(\operatorname{Id}+\nabla\chi(x))(i,j)\pi(dx)=tD(i,j).
\end{align*} 
This is a consequence of \cref{lem:ergodic-thm}.\\ 
The boundary term $\frac{1}{\sqrt{n}}[\chi(X_{0}^{\mathbb{T}^{d}})-\chi(X_{nt}^{\mathbb{T}^{d}})]$ vanishes when $n\to\infty$ as $\chi\in L^{\infty}(\mathbb{T}^{d})$. Furthermore, as a processes, 
\begin{align*}
\paren[\bigg]{\frac{1}{\sqrt{n}}[\chi(X_{0}^{\mathbb{T}^{d}})-\chi(X_{nt}^{\mathbb{T}^{d}})]}_{t\in [0,T]}
\end{align*} converges to the constant zero process almost surely with respect to the uniform topology in $C([0,T],\R^{d})$.\\ Using Slutsky's lemma and combining with the functional martingale central limit theorem above,
we obtain weak convergence of $(n^{-1/2}X_{nt})_{t\in [0,T]}$ to the Brownian motion $\sqrt{D}W$ with the constant diffusion matrix $D$ stated in the theorem.\\

\noindent Let now $\alpha\in (1,2)$. We rescale by $n^{-1/\alpha}$ and claim that the martingale $n^{-1/\alpha}M_{nt}$ vanishes in $L^{2}(\p)$ for $n\to\infty$. 
Indeed, in this case the martingale $M$ is given by 
\begin{align*}
M_{t}=\int_{0}^{t}\int_{\R^{d}\setminus \{0\}}[\chi(X_{s-}^{\mathbb{T}^{d}}+[y])-\chi(X_{s-}^{\mathbb{T}^{d}})]\hat{\pi}(ds,dy).
\end{align*} 
Using the estimate from \cite[Lemma 8.22]{peszat_zabczyk_2007} and the mean-value theorem, we obtain
\begin{align*}
\E[\sup_{t\in[0,T]}\abs{M_{nt}}^{2}]&\lesssim \int_{0}^{nT}\int_{\R^{d}\setminus\{0\}}\E[\abs{\chi(X_{s-}^{\mathbb{T}^{d}}+[y])-\chi(X_{s-}^{\mathbb{T}^{d}})}^{2}]\mu(dy)ds\\&=\int_{0}^{nT}\int_{\R^{d}\setminus\{0\}}\E[\abs{\chi(X_{s}^{\mathbb{T}^{d}}+[y])-\chi(X_{s}^{\mathbb{T}^{d}})}^{2}]\mu(dy)ds
\allowdisplaybreaks
\\&\leqslant \int_{0}^{nT}\int_{B(0,1)^{c}}\E[\abs{\chi(X_{s}^{\mathbb{T}^{d}}+[y])-\chi(X_{s}^{\mathbb{T}^{d}})}^{2}]\mu(dy)ds\\&\qquad+ \int_{0}^{nT}\int_{B(0,1)\setminus\{0\}}\E[\abs{\chi(X_{s}^{\mathbb{T}^{d}}+[y])-\chi(X_{s}^{\mathbb{T}^{d}})}^{2}]\mu(dy)ds\\&\leqslant 2nT\mu(B(0,1)^{c})\norm{\chi}_{L^{\infty}(\mathbb{T}^{d})^{d}}^{2}+2nT\norm{\nabla\chi}_{L^{\infty}(\mathbb{T}^{d})^{d\times d}}^{2}\int_{B(0,1)\setminus\{0\}}\abs{y}^{2}\mu(dy)\\&\lesssim nT.
\end{align*}
Hence, we conclude 
\begin{align}\label{eq:M-bound}
\E[\sup_{t\in[0,T]}\abs{n^{-1/\alpha}M_{nt}}^{2}]\lesssim Tn^{1-2/\alpha}
\end{align} and since $\alpha<2$, we obtain the claimed convergence to zero.\\ As the $J_{1}$-metric (for definition, see \cite[Chapter VI, Equation 1.26]{Jacod2003}) can be bounded by the uniform norm, \eqref{eq:M-bound} implies in particular, that the process $(n^{-1/\alpha}M_{nt})_{t\in[0,T]}$ converges to the constant zero process in probability with respect to the $J_{1}$-topology on the Skorokhod space $D([0,T],\R^{d})$. Furthermore, $(n^{-1/\alpha}L_{nt})_{t\geqslant 0}\stackrel{d}{=}(L_{t})_{t\geqslant 0}$. Using \cite[Chapter VI, Proposition 3.17]{Jacod2003} and that the constant process is continuous, we thus obtain that $(n^{-1/\alpha}X_{nt})_{t\geqslant 0}$ convergences in distribution in $D([0,T],\R^{d})$ to the $\alpha$-stable process $(\tilde{L}_{t})_{t\in [0,T]}$, that has the same law as $(L_{t})_{t\in[0,T]}$. 
\end{proof}

\noindent Utilizing the correspondence of the solution of the SDE (i.e.~the solution of the martingale problem) to the parabolic generator PDE via Feynman-Kac, we can now show the corresponding periodic homogenization result for the PDE as a corollary.
\begin{corollary}\label{cor:PDEhomog}
Let $F$ and $F^{\R^{d}}$ be as in \cref{thm:main-thm-ph}. Assume moreover that $\langle F\rangle_{\pi}=0$ and let $f\in C_{b}(\R^{d})$. Let $T>0$ and let $u\in D_{T}=\{u\in C_{T}\calC^{\alpha+\beta}\cap C^{1}_{T}\calC^{\beta}\mid u^{\sharp}:=u-\nabla u\para I(F)\in C_{T}\calC^{2(\alpha+\beta)-1}\cap C^{1}_{T}\calC^{\alpha+2\beta-1}\}$ with $I_{t}(f):=\int_{0}^{t}P_{t-s}(f)ds$ be the mild solution of the singular parabolic PDE
\begin{align*}
(\partial_{t}-\mathfrak{L})u=0,\quad u_{0}=f^{\epsilon},
\end{align*} where $f^{\epsilon}(x):=f(\epsilon x)$. Let $u^{\epsilon}(t,x):=u(\epsilon^{-\alpha}t,\epsilon^{-1} x)$ with $u^{\epsilon}(0,\cdot)=f$.\\ 
Let furthermore, for $\alpha=2$ and $-\La=\frac{1}{2}\Delta$, $\overline{u}$ be the solution of 
\begin{align*}
(\partial_{t}-D:\nabla\nabla)\overline{u}=0,\quad \overline{u}_{0}=f,
\end{align*} with notation $D:\nabla\nabla:=\sum_{i,j=1,...,d}D(i,j)\partial_{x_{i}}\partial_{x_{j}}$,\\  and for $\alpha\in (1,2)$, let $\overline{u}$ be the solution of
\begin{align*}
(\partial_{t}+\La)\overline{u}=0,\quad \overline{u}_{0}=f.
\end{align*} 
Then, for any $t\in (0,T]$, $x\in\R^{d}$, we have the convergence $u^{\epsilon}_{t}(x)\to \overline{u}_{t}(x)$ for $\epsilon\to 0$.
\end{corollary}
\begin{remark}
Note that $u^{\epsilon}$ solves $(\partial_{t}-\mathfrak{L}^{\epsilon})u^{\epsilon}=0$, $u^{\epsilon}_{0}=f$ with operator $\mathfrak{L}^{\epsilon}g=-\La g+\epsilon^{1-\alpha} F(\epsilon^{-1} \cdot)\nabla g$.
\end{remark}
\begin{remark}
If $\alpha=2$ and $F$ is of gradient-type, that is, $F=\nabla f$ for $f\in\calC^{1+\beta}$ ($f$ is a continuous function, as $1+\beta>0$), the invariant measure is explicitly given by $d\pi=c^{-1}e^{-f(x)}dx$ with suitable normalizing constant $c>0$, since the operator is of divergence form, $\mathfrak{L}=e^{f}\nabla \cdot(e^{-f}\nabla\cdot)$. Then it follows that $\langle F\rangle_{\pi}=\int_{\mathbb{T}^{d}}\nabla e^{-f(x)}dx=0$. Thus, $F$ satisfies the assumptions of \cref{cor:PDEhomog}.
\end{remark}
\begin{proof}[Proof of \cref{cor:PDEhomog}] 
Notice that $(\tilde{u}_{s}:=u_{t-s})_{s\in[0,t]}$ solves the backward Kolmogorov equation $(\partial_{s}+\mathfrak{L})\tilde{u}=0, \tilde{u}(t,\cdot)=f^{\epsilon}$.
Approximating $f$ by $\calC^{3}(\R^{d})$ functions and using that $X$ solves the $(\partial_{t}+\mathfrak{L},x)$-martingale problem, we obtain 
\begin{align*}
u^{\epsilon}(t,x)=\E_{X_{0}=\epsilon^{-1}x}[f(\epsilon X_{\epsilon^{-\alpha}t})].
\end{align*} 
The stated convergence then follows from \cref{thm:main-thm-ph}. Indeed, if $X_{0}=\epsilon^{-1}x$, then $\epsilon X_{\epsilon^{-2}\cdot}\to W^{x}$ in distribution, where $W^{x}$ is the Brownian motion started in $x$ with covariance $D$, respectively $\epsilon X_{\epsilon^{-\alpha}\cdot}\to L^{x}$ if $\alpha\in (1,2)$ for the $\alpha$-stable process $L$ with generator $(-\La)$ and $L_{0}=x$. The Feynman-Kac formula for the limit process then gives that the limit of $(u^{\epsilon}(t,x))$ equals $\overline{u}(t,x)=\E[f(W^{x})]$ if $\alpha=2$, respectively $\overline{u}(t,x)=\E[f(L^{x})]$ if $\alpha\in (1,2)$.
\end{proof}

\begin{remark}[Brox diffusion with Lévy noise]\label{rem:p-Brox}
We can apply our theory to obtain the long-time behaviour of the periodic Brox diffusion with Lévy noise (see \cite{kp} for the construction).
As $\alpha\in (1,2]$, \cref{thm:main-thm-ph} yields that $\abs{X_{t}}\sim t^{1/\alpha}$ for $t\to\infty$.\\
In the non-periodic situation, the long-time behaviour of the Brox diffusion with Brownian noise is however very different. Brox \cite{Brox1986} proved, that the diffusion gets trapped in local minima of the white noise environment and thus slowed down (that is, for almost all environments: $\abs{X_{t}}\sim \log(t)^{2}$ for $t\to\infty$, cf. \cite[Theorem 1.4]{Brox1986}). In the non-periodic pure stable noise case, the long-time behaviour of the Brox diffusion is an open problem, that we leave for future research. 
\end{remark}

\end{section}

\appendix
\begin{section}{Appendix}\label{Appendix A}
\begin{proof}[Proof of \cref{lem:periodic-semi-est}]
To show \eqref{eq:p-la}, we notice that by the isometry of the spaces $L^{2}(\mathbb{T}^{d})$, $l^{2}(\mathbb{Z}^{d})$ by the Fourier transform,
\begin{align*}
\norm{\Delta_{j}\La u}_{L^{2}(\mathbb{T}^{d})}^{2}=\sum_{k\in\mathbb{Z}^{d}}\abs{\rho_{j}(k)\psi^{\alpha}_{\nu}(k)\hat{u}(k)}^{2}.
\end{align*}
Due to $\rho_{j}(k)\neq 0$ only if $\abs{k}\sim 2^{j}$ and $\abs{\psi^{\alpha}_{\nu}(k)}\lesssim\abs{k}^{\alpha}$, we obtain that
\begin{align*}
\norm{\Delta_{j}\La u}_{L^{2}(\mathbb{T}^{d})}^{2}\lesssim 2^{2j\alpha}\sum_{k\in\mathbb{Z}^{d}}\abs{\rho_{j}(k)\hat{u}(k)}^{2}=2^{2j\alpha}\norm{\Delta_{j}u}_{L^{2}(\mathbb{T}^{d})}^{2}
\end{align*} and thus 
\begin{align*}
\norm{\La u}_{\calC^{\beta-\alpha}_{2}(\mathbb{T}^{d})}=\sup_{j}2^{j(\beta-\alpha)}\norm{\Delta_{j}\La u}_{L^{2}(\mathbb{T}^{d})}\lesssim\sup_{j} 2^{j\beta}\norm{\Delta_{j}u}_{L^{2}(\mathbb{T}^{d})}=\norm{u}_{\calC^{\beta}_{2}(\mathbb{T}^{d})}.
\end{align*}
To show \eqref{eq:p-semi}, we again use the isometry, such that
\begin{align*}
\norm{\Delta_{j} P_{t}u}_{L^{2}(\mathbb{T}^{d})}^{2}=\sum_{k\in\mathbb{Z}^{d}}\abs{\rho_{j}(k)\exp(-t\psi^{\alpha}_{\nu}(k))\hat{u}(k)}^{2}.
\end{align*}
For $j=-1$, $\rho_{j}$ is supported in a ball around zero and as $\abs{\exp(-t\psi^{\alpha}_{\nu}(k))}\leqslant 1$, the estimate $\norm{\Delta_{j} P_{t}u}_{L^{2}(\mathbb{T}^{d})}^{2}\lesssim (t^{-\theta/\alpha}\vee 1)2^{\theta}\sum_{k\in\mathbb{Z}^{d}}\abs{\rho_{j}(k)\hat{u}(k)}^{2}$ holds trivially for $\theta\geqslant 0$. For $j>-1$, $p_{j}$ is supported away from zero and we can use that $\exp(-t\psi^{\alpha}_{\nu}(\cdot))$ is a Schwartz function away from $0$ and thus, for $\abs{k}>0$, $\abs{\exp(-t\psi^{\alpha}_{\nu}(k))}\lesssim (t\psi^{\alpha}_{\nu}(k)+1)^{-\theta/\alpha}\lesssim t^{-\theta/\alpha}\abs{k}^{-\theta}$, for any $\theta\geqslant 0$. Thus, for $j>-1$, we obtain
\begin{align*}
\norm{\Delta_{j} P_{t}u}_{L^{2}(\mathbb{T}^{d})}^{2}\leqslant 2^{-2j\theta}t^{-\theta/\alpha}\sum_{k\in\mathbb{Z}^{d}}\abs{\rho_{j}(k)\hat{u}(k)}^{2}=2^{-2j\theta}t^{-\theta/\alpha}\norm{\Delta_{j}u}_{L^{2}(\mathbb{T}^{d})}^{2},
\end{align*} such that together \eqref{eq:p-semi} follows. To obtain the remaining estimate, we argue in a similar manner using that, due to Hölder-continuity of the exponential function, for $\theta/\alpha\in [0,1]$,  $\abs{\exp(-t\psi^{\alpha}_{\nu}(k))-1}\leqslant\abs{t\psi^{\alpha}_{\nu}(k)}^{\theta/\alpha}\leqslant t^{\theta/\alpha}\abs{k}^{\theta}$.
\end{proof}
\begin{proof}[Proof of \cref{lem:exp-schauder}]
By the assumption of vanishing zero-order Fourier mode, we have 
\begin{align*}
P_{t}g=\sum_{\abs{k}\geqslant 1}\exp(-t\psi^{\alpha}_{\nu}(k))\hat{g}(k)e_{k}.
\end{align*} 
Thus, we obtain by $\psi^{\alpha}_{\nu}(k)\geqslant c\abs{k}^{\alpha}$ for some $c>0$ (follows from \cref{ass}) the trivial estimate
\begin{align*}
\norm{\Delta_{j}(P_{t}g)}_{L^{2}(\mathbb{T}^{d})}= \sum_{\abs{k}\geqslant 1}\abs{p_{j}(k)\exp(-t\psi^{\alpha}_{\nu}(k))\hat{g}(k)}^{2}\leqslant \norm{g}_{\calC^{\beta}_{2}(\mathbb{T}^{d})}2^{-j\beta}\exp(-tc).
\end{align*} 
Together with \cref{lem:periodic-semi-est}, we then obtain for any $\theta\geqslant 0$,
\begin{align*}
\norm{\Delta_{j}(P_{t}g)}_{L^{2}(\mathbb{T}^{d})}\lesssim \norm{g}_{\calC^{\beta}_{2}(\mathbb{T}^{d})}\min\paren[\big]{2^{-j\beta}\exp(-tc),\, 2^{-j(\beta+\theta)}(t^{-\theta/\alpha}\vee 1)}.
\end{align*} 
The claim thus follows by interpolation.
\end{proof}
\end{section}

\section*{Acknowledgements}

H.K.~is supported by the Austrian Science Fund (FWF) Stand-Alone programme P 34992. Part of the work was done when H.K. was employed at Freie Universität Berlin and funded by the DFG under Germany's Excellence Strategy - The Berlin Mathematics Research Center MATH+ (EXC-2046/1, project ID: 390685689). N.P.~ gratefully acknowledges financial support by the DFG via Research Unit FOR2402 and through the grant CRC 1114 "Scaling Cascades in Complex Systems".

\end{document}